\documentclass[11pt]{amsart}
\pdfoutput=1
\usepackage{fullpage}
\usepackage{amsfonts, amsmath, amsthm, amssymb, amstext, graphicx, dsfont, stmaryrd, extarrows,amsthm}
\usepackage[T1]{fontenc}
\usepackage{textcomp}
\usepackage{lmodern}
\usepackage{microtype} 
\usepackage{amscd}
\usepackage{mathrsfs}
\usepackage[colorinlistoftodos]{todonotes}
\usepackage{color}
\usepackage{bbm}
\usepackage{verbatim}
\usepackage[mathscr]{euscript}
\usepackage{mathtools} 
\usepackage{mathrsfs} 

\usepackage[english]{babel}
\usepackage[utf8]{inputenc}
\usepackage{lmodern}
\usepackage[shortlabels]{enumitem}

\usepackage{nicefrac} 
\usepackage{textcomp}
\usepackage{transparent} 
\usepackage{setspace} 
\usepackage[autostyle,italian=guillemets]{csquotes}

\usepackage[style=alphabetic,backend=biber]{biblatex}
\usepackage{hyperref}
\hypersetup{
	colorlinks=true,
	citecolor=blue,
	linkcolor=blue,
	filecolor=magenta,      
	urlcolor=cyan,
}

\usepackage{tikz}
\usepackage{tikz-cd}
\usetikzlibrary{arrows,calc,matrix}
\usepackage{booktabs, hhline} 

\numberwithin{equation}{section}
\theoremstyle{plain}
\newtheorem{theorem}[equation]{Theorem}
\newtheorem*{thm*}{Theorem}

\newtheorem{proposition}[equation]{Proposition}
    
\newtheorem{corollary}[equation]{Corollary}       
\newtheorem{lemma}[equation]{Lemma}

\theoremstyle{definition} 
\newtheorem{definition}[equation]{Definition} 
\newtheorem{hyp}[equation]{Hypothesis}
\newtheorem{ex}[equation]{Example}

\newtheorem{rem}[equation]{Remark}   

\newtheorem{nota}[equation]{Notation}

\newtheorem{choice}[equation]{Choice}

\allowdisplaybreaks[1]

\makeatletter 
\renewenvironment{proof}[1][\proofname]{%
  \par\pushQED{\qed}\normalfont%
  \topsep6\p@\@plus6\p@\relax
  \trivlist\item[\hskip\labelsep\bfseries#1\@addpunct{.}]%
  \ignorespaces
}{%
  \popQED\endtrivlist\@endpefalse
}
\makeatother

\def\bigquotient#1#2{%
    \raise1ex\hbox{$#1$}\Big/\lower1ex\hbox{$#2$}%
}

\def\quotient#1#2{\mathchoice
{\raisebox{.3ex}{$\mathsurround=0pt\displaystyle #1$}\mkern -1mu/\mkern -1mu\raisebox{-.5ex}{$\mathsurround=0pt\displaystyle #2$}}
{\raisebox{.3ex}{$\mathsurround=0pt\textstyle #1$}\mkern -1mu/\mkern -1mu\raisebox{-.5ex}{$\mathsurround=0pt\textstyle #2$}}
{\raisebox{.1ex}{$\mathsurround=0pt\scriptstyle #1$}\mkern -1mu/\mkern -1mu\raisebox{-.3ex}{$\mathsurround=0pt\scriptstyle #2$}}
{\raisebox{.1ex}{$\mathsurround=0pt\scriptscriptstyle #1$}\mkern -1mu/\mkern -1mu\raisebox{-.1ex}{$\mathsurround=0pt\scriptscriptstyle #2$}}}

\usepackage{letltxmacro}
\LetLtxMacro\orgvdots\vdots
\LetLtxMacro\orgddots\ddots

\makeatletter
\newcommand{\colim@}[2]{%
  \vtop{\m@th\ialign{##\cr
    \hfil$#1\operator@font colim$\hfil\cr
    \noalign{\nointerlineskip\kern1.5\ex@}#2\cr
    \noalign{\nointerlineskip\kern-\ex@}\cr}}%
}
\newcommand{\colim}{%
  \mathop{\mathpalette\colim@{\rightarrowfill@\scriptscriptstyle}}\nmlimits@
}
\renewcommand{\varprojlim}{%
  \mathop{\mathpalette\varlim@{\leftarrowfill@\scriptscriptstyle}}\nmlimits@
}
\renewcommand{\varinjlim}{%
  \mathop{\mathpalette\varlim@{\rightarrowfill@\scriptscriptstyle}}\nmlimits@
}
\makeatother

\tikzset{%
	symbol/.style={%
		draw=none,
		every to/.append style={%
			edge node={node [sloped, allow upside down, auto=false]{$#1$}}}
	}
}
\newcommand{\Q}{\mathbb{Q}}

\newcommand{\Z}{\mathbb{Z}}
\newcommand{\C}{\mathbb{C}}
\newcommand{\T}{\mathbb{T}}

\newcommand{\A}{\mathcal{A}}

\renewcommand{\c}{\mathcal{C}}

\newcommand{\F}{\mathcal{F}}

\newcommand{\E}{\mathcal{E}}
\newcommand{\K}{\mathcal{K}}

\renewcommand{\O}{\mathcal{O}}

\newcommand{\I}{\mathcal{I}}
\newcommand{\inj}{\mathbb{I}}

\renewcommand{\epsilon}{\varepsilon}
\renewcommand{\phi}{\varphi}

\newcommand{\modulo}[1]{\left\lvert#1\right\rvert}
\newcommand{\gen}[1]{\langle #1 \rangle}

\renewcommand{\i}[1]{\iota_{#1}}

\newcommand{\catname}[1]{{\normalfont\textbf{#1}}}

\newcommand{\oef}[1]{\O_{\quotient{\F}{#1}}}
\newcommand{\oefh}[1]{\O_{\quotient{\F}{H_{#1}}}}
\newcommand{\of}{\O_{\F}}
\newcommand{\tens}[1]{%
	\mathbin{\mathop{\otimes}\displaylimits_{#1}}%
}

\newcommand{\eginv}{\E_G^{-1}}
\newcommand{\eh}[1]{\E_{H_{#1}}}
\newcommand{\ehinv}[1]{\E_{H_{#1}}^{-1}}
\newcommand{\egh}[1]{\E_{G/H_{#1}}}
\newcommand{\eghinv}[1]{\E_{G/H_{#1}}^{-1}}
\newcommand{\einv}[1]{\E_{#1}^{-1}}

\newcommand{\fix}[1]{\Phi^{#1}}

\newcommand{\pai}[1]{\pi_*^{#1}}
\newcommand{\defp}{DE\F_+}
\newcommand{\defhp}[1]{DE\F/{#1}_+}

\renewcommand{\sp}[1]{#1\Spectra}

\newcommand{\Sp}{\catname{Sp}}

\DeclareMathOperator{\Hom}{Hom}
\DeclareMathOperator{\Div}{Div}
\DeclareMathOperator{\Ker}{Ker}

\DeclareMathOperator{\Dim}{dim}
\DeclareMathOperator{\Codim}{codim}
\DeclareMathOperator{\Id}{Id}

\DeclareMathOperator{\supp}{supp}
\DeclareMathOperator{\Spec}{Spec}
\DeclareMathOperator{\Spectra}{-Spectra}
\DeclareMathOperator{\Ext}{Ext}

\DeclareMathOperator{\im}{Im}

\DeclareMathOperator{\Mod}{Mod}
\DeclareMathOperator{\Shv}{Shv}
\DeclareMathOperator{\Fib}{Fib}

\DeclareMathOperator{\KQ}{KQ}

\newcommand{\ox}{\O_{\X}}

\newcommand{\odij}{\O_{D_{ij}}}
\newcommand{\oxx}[1]{\O_{\X,#1}}

\newcommand{\tp}[1]{#1^{\text{TP}}}
\DeclareMathOperator{\CC}{CC}
\renewcommand{\div}[1]{D_{#1,n_{#1}}}
\newcommand{\hbg}[1]{H^*(BG/{#1})}

\newcommand{\h}[2]{\mathcal{H}^{#1}_{#2}}
\newcommand{\dij}{{D_{ij}}}
\newcommand{\dip}{{D_{i,P}}}
\newcommand{\tij}{{t_{ij}}}
\newcommand{\cij}{{c_{ij}}}
\newcommand{\hatij}{{\hat t_{ij}}}

\newcommand{\zar}[1]{#1^{\text{Zar}}}

\newcommand{\ov}[1]{\overline {#1}}

\newcommand{\susp}[1]{\Sigma^{#1}}
\newcommand{\odv}{\O(-D_V)}

\newcommand{\mij}{{m_{ij}}}
\newcommand{\tin}[1]{t_{#1, n_{#1}}}

\newcommand{\ha}{H_A^{n_A}}
\newcommand{\hb}{H_B^{n_B}}
\newcommand{\xa}{x_A}
\newcommand{\xb}{x_B}
\newcommand{\za}{z_A^{n_A}}
\newcommand{\zb}{z_B^{n_B}}
\newcommand{\hatone}{{\hat t_{A,n_A}}}
\newcommand{\hattwo}{{\hat t_{B,n_B}}}
\newcommand{\pia}{\pi_A^{n_A}}
\newcommand{\pib}{\pi_B^{n_B}}
\newcommand{\ta}{{t_{A,n_A}}}
\newcommand{\tb}{{t_{B,n_B}}}
\newcommand{\tijvij}{t^{v_{ij}}_{ij}}

\newcommand{\X}{\chi}
\newcommand{\XF}{\mathfrak{X}}
\newcommand{\var}[1]{\XF(#1)}
\newcommand{\bvar}[1]{\bar{\XF}(#1)}

\newcommand{\etinv}{\E_{\T^2}^{-1}}

\newcommand{\ethinv}[1]{\E_{\T^2/H_{#1}}^{-1}}
\newcommand{\hbti}[1]{H^*(B\T^2/{H_i^{#1}})}
\newcommand{\hbt}[1]{H^*(B\T^2/{#1})}

\title{An algebraic model for rational $\T^2$-equivariant elliptic cohomology}
\author{Matteo Barucco}
\address{Department of Mathematics, University of Warwick}
\email{Matteo.Barucco@Warwick.ac.uk}
\begin{document}
\begin{abstract}
	We construct a rational $\T^2$-equivariant elliptic cohomology theory for the $2$-torus $\T^2$, starting from an elliptic curve $\c$ over $\C$ and a coordinate data around the identity. The theory is defined by constructing an object $E\c_{\T^2}$ in the algebraic model category $d\A(\T^2)$, which by Greenlees and Shipley \cite{john:torusship} is Quillen-equivalent to rational $\T^2$-spectra. This result is a generalisation to the $2$-torus of the construction \cite{john:elliptic} for the circle $\T$. The object $E\c_{\T^2}$ is directly built using geometric inputs coming from the Cousin complex of the structure sheaf of $\c\times \c$.
\end{abstract}
\maketitle

\setcounter{tocdepth}{1}
\tableofcontents
\section{Introduction}
\subsection{Motivation}
    One of the first applications of the algebraic models for rational $\T$-spectra ($\T$ is the circle group) introduced by Greenlees \cite{john:mem} was the construction of an algebraic model for rational $\T$-equivariant elliptic cohomology \cite{john:elliptic}. This result was part of a long quest started by Grojnowski \cite{groj:delocalised} for a satisfactory and comprehensive construction of equivariant elliptic cohomology. Equivariant elliptic cohomology is already a deep subject that shares an interesting bond with physics (see for example the Stolz-Teichner program \cite{stolz}, or work of Berwick-Evans \cite{daniel:supersymmetric}). The original motivation of Grojnowski was to construct certain elliptic algebras, but the subject found numerous applications to representation theory (for $G$ connected compact Lie group), and for organizing moonshine phenomena (especially for $G$ finite). We refer to the introduction of \cite{daniel:geometric} and \cite{gepner:onequiv} for a more detailed account of the story of this subject. This quest peaked in recent years with the work of Lurie \cite{lurie:ellipticI}, \cite{lurie:ellipticII}, \cite{lurie:ellipticIII}, and Gepner and Meier \cite{gepner:onequiv} who defined an integral theory of equivariant elliptic cohomology for any compact Lie group, entirely in the setting of spectral algebraic geometry. A comparison theorem between the construction we present in this paper and the construction of Lurie, Gepner and Meier will be interesting to pursue and could potentially lead to explicit computations of the Lurie-Gepner-Meier theory in the rational case. This is beyond the scope of this paper, and even for the circle case unpublished work of Gepner and Greenlees shows it is a substantial task.
    
    Essential invariants of $G$-spaces are $G$-equivariant cohomology theories: equivariant $K$-theory, equivariant cobordism, Bredon and Borel cohomology just to name a few examples. For $G$ a compact Lie group, via stabilisation one can construct a category of $G$-spectra where every such cohomology theory $E_G^*(\_)$ is represented by a $G$-spectrum $E$, in the sense that for any based $G$-space $X$ we have $E_G^*(X)=[\Sigma^{\infty}X, E]^G_*$. The gain with the category of $G$-spectra is in the structure, in particular one can do homotopy theory in it. If we restrict our attention to rational $G$-equivariant cohomology theories to simplify the situation, they are represented by rational $G$-spectra: the localization of $G$-spectra at $H\Q$. 
    
    Algebraic models are a useful tool to study rational equivariant cohomology theories. The main idea is to define an abelian category $\A(G)$ for $G$ a compact Lie group and an homology functor
    \begin{equation}
    	\label{eq:homologyfunctor}
    	\pi_*^{\A}:\sp G_{\Q} \to \A(G)
    \end{equation}
    equipped with an Adams spectral sequence to compute maps in rational $G$-spectra, and therefore the values of the theory:
    \begin{equation}
    	\label{eq:ass}
    	\Ext^{*,*}_{\A}(\pi_*^\A(X),\pi_*^\A(Y)) \Longrightarrow [X,Y]^G_{*}.
    \end{equation}
    For certain groups the homology functor $\pi_*^{\A}$ can be lifted to a Quillen-equivalence 
    \begin{equation}
    	\label{eq:Quillen}
    	\sp G_\Q \simeq_Q d\A(G)
    \end{equation}
    between rational $G$-spectra and the category $d\A(G)$ of differential graded objects in $\A(G)$. When this happens algebraic models can also be used to build rational $G$-equivariant cohomology theories simply constructing objects in $d\A(G)$. This is exactly the method used in \cite{john:elliptic}: building an object $E\c_{\T}$ in $d\A(\T)$ and using the Quillen-equivalence \eqref{eq:Quillen} for the circle $G=\T$ to define a $\T$-equivariant elliptic cohomology theory. 
    
    Greenlees and Shipley have extended the Quillen-equivalence \eqref{eq:Quillen} to tori of any rank $\T^r$ \cite[Theorem 1.1]{john:torusship}. It is therefore natural to try to generalize the construction of \cite{john:elliptic} to higher dimensional tori. The first step of this project is to build an object $E\c_{\T^2}\in d\A(\T^2)$ representing $\T^2$-equivariant elliptic cohomology, which is precisely the intended purpose of this paper. The $\T^2$ case is somehow separated from the general $\T^r$ case. Namely the Adams spectral sequence \eqref{eq:ass} collapses at the second page for $\T^2$ resulting in a neat and explicit description of the values of the theory on spheres of complex representations \eqref{eq:valuetheory}. Even if a similar description is expected to be true for higher dimensional tori, the Adams spectral sequence is not expected to collapse at the second page, and a substantial study of it may be necessary. Moreover the construction for $\T^2$ is complicated enough to shed some light over compatibility constraints among connected subgroups of the same codimension, not visible in the circle case (like for example the use of completed coordinates needed for Lemma \ref{lem:ccmap}). At the same time the situation is still simple enough to allow explicit visualization of the objects and to avoid use of combinatorics and inductive arguments necessary for higher tori, that would complicate the comprehension of the main ideas of the construction.
    
    The point of this method is that the construction of an object in $d\A(\T^2)$ so closely corresponds to the algebra of functions over the algebraic variety $\X:=\c\times \c$, where $\c$ is our fixed elliptic curve. The contact point is the Cousin complex (as introduced by Grothendieck \cite[Proposition 2.3]{hartshorne:residues}) of the structure sheaf $\O_{\X}$. From the algebraic geometry side this Cousin complex computes the cohomology of coherent sheaves over $\X$, since it is a flabby resolution of $\O_{\X}$. While in $\A(\T^2)$ this Cousin complex clearly matches the terms of an injective resolution of our object $E\c_{\T^2}$, and therefore computes the values $E\c_{\T^2}(\_)$ via the Adams spectral sequence. As a consequence calculations of the theory $E\c_{\T^2}(\_)$ are directly reduced to the cohomology of sheaves of the algebraic variety $\X$. Fundamental examples are spheres of complex representations: given a complex $\T^2$-representation $V$ with no fixed points we compute the values $E\c_{\T^2}(S^V)$ over the one point compactification $S^V$, and check it corresponds to the cohomology of the coherent sheaf $\O(-D_V)$, for a certain divisor $D_V$ associated with $V$.
    
    Two main subjects can benefit from our construction of $E\c_{\T^2}$. From the algebraic models perspective, this is the first non-trivial explicit construction of an object in $d\A(G)$ for higher dimensional tori, that is directly built in the algebraic model and does not come from a spectrum through the homology functor $\pi_*^{\A}$. This is interesting since it is a first example of use of these algebraic models for higher tori as a building tool for new theories. Many steps of the construction can be replicated with different geometric inputs to potentially define new and interesting rational equivariant cohomology theories. From the elliptic cohomology perspective, on top of giving a definition in the rational $\T^2$-case closer to the actual geometry of $\c$, it can also enhance computations in the rational $\T$-equivariant case. For example proving a split condition between the construction of Greenlees $E\c_{\T}$ and our construction of $E\c_{\T^2}$ will help towards this direction. The idea is to view a $\T$-space $X$ as the quotient of a $\T^2$-space $Y$ by the action of the subgroup $\T\times \{1\}$. If we can compute $E\c_{\T^2}^*(Y)$ using the computation on spheres here presented, than using the splitting condition we will also know $E\c_{\T}^*(X)$.
    
    \subsection{The main result}
    The construction for the circle \cite[Theorem 1.1]{john:elliptic} works as follows. Greenlees starts from an elliptic curve $\c$ over a $\Q$-algebra with a choice of coordinate $t_e\in \O_{\c,e}$ vanishing to the first order at $e$, and with poles only at points of finite order of $\c$. From this he constructs uniformizers $t_n$ for the points of exact order $n$: $\c\gen n$, and also changes the topology so that $\c\gen n$ are the irreducible closed subsets. The ring of meromorphic functions in this new topology, together with the local rings at the various $\c\gen n$ are then used as geometric input to build the object $E\c_{\T}\in d\A(\T)$, where the $t_n$ are used to give the actions of the appropriate rings. After building $E\c_{\T}$, using the Adams spectral sequence he computes the values of the associated cohomology theory on spheres of complex representations $E\c_{\T}^*(S^V)$, showing it has all the right properties to be called rational $\T$-equivariant elliptic cohomology.
    
    For higher dimensional tori we also have an algebraic model $\A(\T^r)$ described in \cite{john:torusI} and \cite{john:torusII}, and the Quillen-equivalence \eqref{eq:Quillen} \cite[Theorem 1.1]{john:torusship}. Therefore we can use the same technique to construct a rational $\T^2$-equivariant elliptic cohomology theory: build an object $E\c_{\T^2}\in d\A(\T^2)$ and with the Adams spectral sequence compute its values on spheres of complex representations.
    
    To be more precise exactly as in \cite{john:elliptic} we start from the data of an elliptic curve $\c$ over the complex numbers and a coordinate $t_e\in \O_{\c,e}$ in the local ring at the identity of $\c$, vanishing to first order at $e$. From this we build our object $E\c_{\T^2}$: this is the main theorem of this article. The construction of the object can be found in Section \ref{sec:4} while the computation on spheres is Theorem \ref{thm:computationspheres}.   
    \begin{theorem}
    \label{thm:fundconstr}
    For every elliptic curve $\c$ over $\C$ and coordinate $t_e\in \O_{\c,e}$, there exists a formal differential graded object $E\c_{\T^2}\in d\A(\T^2)$ whose associated rational $\T^2$-equivariant cohomology theory $E\c_{\T^2}^*(\_)$ is 2-periodic. The value on the one point compactification $S^V$ for a complex $\T^2$-representation $V$ with no fixed points is given in terms of the sheaf cohomology of a line bundle $\odv$ over $\X=\c\times \c$:
    \begin{equation}
    	\label{eq:valuetheory}
    	E\c_{\T^2}^n(S^V)\cong 
    	\begin{cases*}
    		H^0(\X, \odv)\oplus H^2(\X, \odv) & $n$ even \\
    		H^1(\X, \odv) & $n$ odd.
    	\end{cases*}
    \end{equation}
	\end{theorem}

	The associated divisor $D_V$ is defined as follows. First our elliptic curve $\c$ defines a functor $\XF$ from compact abelian Lie groups to complex manifolds. If $H$ is a compact abelian Lie group and $\c$ our fixed elliptic curve, define
		\[
		\XF(H):=\Hom_{\text{Ab}}(H^*,\c).
		\]
	Where we are considering group homomorphisms, and $H^*:= \Hom_{\text{Lie}}(H, \T)$ is the character group of $H$. Note this is an exact functor inducing an embedding $\XF(K)\hookrightarrow \XF(H)$ for every embedding $K\hookrightarrow H$.
	The fundamental algebraic surface we are working on is $\X:=\XF(\T^2)\cong \c\times \c$. Note $\XF(H)$ is a subvariety of $\X$ of the same dimension of $H$. For every non-trivial character $z$ of $\T^2$ consider the codimension $1$ subgroup $\Ker(z)$, then $\XF(\Ker(z))$ is a codimension $1$ subvariety of $\X$ that is the associated divisor of $z$. If $V=\sum_z\alpha_zz$ define
	\[
	D_V:= \sum_z \alpha_z\XF(\Ker(z)),
	\]
	where in both sums $z$ runs over all the non-trivial characters of $\T^2$, and $\alpha_z$ is a non-negative integer.
    
    To be more precise our construction of $E\c_{\T^2}\in d\A(\T^2)$ proceeds as follows. We will build an exact sequence of three injective objects in $\A(\T^2)$:
    \begin{equation}
    	\label{eq:intrchain}
    	\inj_0 \xrightarrow{\phi_0} \inj_1 \xrightarrow{\phi_1} \inj_2\to 0.
    \end{equation}
	In $d\A(\T^2)$ the map $\phi_0$ factors through the fibre of $\phi_1$, and defines a map $\tilde{\phi_0}$ whose fibre we can define to be our object $E\c_{\T^2}$. We will prove from this construction that the object $E\c_{\T^2}$ is formal: namely is quasi-isomorphic to its homology $H_*(E\c_{\T^2})$. Moreover the exact sequence we started with \eqref{eq:intrchain} is an injective resolution in $\A(\T^2)$ of the homology $H_*(E\c_{\T^2})$. In this way in using the Adams Spectral sequence to compute the values of the theory, we will only need the sequence \eqref{eq:intrchain}, so we can build it with the right geometric inputs to force the values on spheres of complex representations stated in \eqref{eq:valuetheory}.
	
	The sequence \eqref{eq:intrchain} is built directly from geometric inputs coming from the elliptic curve. The bridge between the category $\A(\T^2)$ and the variety $\X$ is the functor $\XF$ on subgroups of $\T^2$. On $\X$ we consider a different topology than the Zariski one that we call the torsion point topology $\tp {\X}$: the generating closed subsets of this new topology are the irreducible components of the subvarieties $\XF(H)$ where $H$ runs over the closed subgroups of $\T^2$. With this new topology we associate to every closed subgroup $H$ a closed subset $\bvar H$. The irreducible parts of $\bvar H$ have the same dimension of $H$, if $H\neq K$ have the same dimension then $\bvar H$ and $\bvar K$ are disjoint, and every irreducible closed subset in $\tp {\X}$ is in one of the $\bvar H$. The TP-topology basically copies the poset structure of the closed subgroups of $\T^2$ with a poset of closed subsets on the variety $\X$, and forgets the rest. An essential aspect of this change of topology is that pushforward of Zariski sheaves have the same sheaf cohomology in the TP-topology.
	
	Taking advantage of these two aspects of the TP-topology, we can consider the sheaf Cousin complex $CC^*(\tp \ox)$ \cite[Proposition 2.3]{hartshorne:residues} of the pushforward of the Zariski structure sheaf: $\tp {\ox}$. This is a flabby resolution of $\tp \ox$, whose $n$-th term decomposes as a direct sum over the irreducible closed subsets $C$ of codimension $n$ of constant sheaves with value a local cohomology module $\h nC(\tp \ox)$. In complete analogy the $n$-th term of the sequence \eqref{eq:intrchain} we want to build will encode information of the cohomology theory at the subgroups of codimension $n$. The correspondence $H\to \bvar H$ will come in handy at this point.
	Recall from \cite{john:torusI} that objects in $\A(\T^2)$ are sheaves of modules over the space of closed subgroups of $\T^2$. The object $\inj_n$ will be a direct sum over the closed subgroups $H$ of codimension $n$ of $\T^2$ of constant sheaves with value the local cohomology $\h n{\bvar H}(\tp \ox)$ coming from the Cousin complex. The differentials $\phi_0$, and $\phi_1$ will also come from the differentials in the Cousin complex.
    
    To compute the value of the theory on a sphere of complex representation $S^V$, when $V$ does not have fixed points, we use the Adams Spectral sequence \eqref{eq:ass} \cite[Theorem 1.1]{john:torusI}. Using the injective resolution \eqref{eq:intrchain} computations are directly reduced to the Cousin complex $CC^*(\tp \ox)$ twisted by the coherent sheaf $\odv$: this turns out to be a flabby resolution of $\odv$ giving its cohomology as a result.
    From this construction it is clear that there is a deep connection between injective resolutions of objects in the algebraic category $\A(\T^r)$ and the theory of Cousin complexes. The filtration by codimension of subgroups corresponds to the filtration by codimension of closed subsets in the Cousin complex, and also the decomposition in constant sheaves is common to both and corresponds term by term. This requires further investigation, and suggests that the same method can be used to build other interesting cohomology theories, possibly changing the geometric inputs.

\subsection{Outline of the paper} 
The paper is organized as follows. In Section \ref{sec:1} we introduce the fundamental functor $\XF$ (Definition \ref{def:varfunctor}) going from subgroups of $\T^2$ to algebraic subvarieties of $\X$, we build the fundamental correspondence $H \to \bvar H$ (Definition \ref{def:bvar}) and we prove its properties. The most important property is Lemma \ref{lem:fin} which identifies these subvarieties as the right geometric counterpart to the various subgroups. In Section \ref{sec:2} we define the TP-topology on $\X$ (Definition \ref{def:tp}). The main result of this section is Corollary \ref{cor:coh} stating that the pushforward $\phi_*$
 for the change of topology map is exact on quasi-coherent sheaves. In Definition \ref{def:coord} we construct a set of completed coordinate functions for the varieties $\bvar H$, needed to fit the local cohomology modules into objects of $\A(\T^2)$. 
 In Section \ref{sec:3} we introduce the main algebraic geometry tool we will use: the sheaf Cousin complex \eqref{eq:cc1}. We prove that the Cousin complex of $\tp \ox$ is a flabby resolution of $\tp \ox$ (Corollary \ref{cor:exactness}): the essential step is Lemma \ref{lem:changesupp} that proves that the support of a sheaf in both topologies is the same. 
 Section \ref{sec:4} is the core of the paper: we construct the object $E\c_{\T^2}$ in $d\A(\T^2)$ (Definition \ref{def:mainobj}). First we detail the construction of $E\c_{\T^2}$ starting from \eqref{eq:intrchain} in \ref{subsec:object}, while the rest of the section is devoted in building \eqref{eq:intrchain}. The three injective objects $\inj_0$, $\inj_1$, $\inj_2$ are respectively built in \ref{sub:inj0}, \ref{sub:inj1}, \ref{sub:inj2}, while the two maps $\phi_0$ and $\phi_1$ are defined in \ref{sub:phi0}, \ref{sub:phi1}. The final part \ref{subsec:exact} deals with exactness of \eqref{eq:intrchain}: Lemma \ref{lem:exactness}. 
 The final section of the paper is dedicated to computing the values of the theory on spheres of complex representations. Here the essential Theorem is Theorem \ref{thm:ext} that computes the second page of the Adams spectral sequence to be the cohomology of the sheaf $\odv$, forging the direct link of the theory with the geometry of $\c$, we conclude the section extending the computations for virtual negative complex representations \eqref{eq:negativevalues}. We end with an Appendix on algebraic models recalling from the literature the definition of the category $\A(G)$, and especially injective objects in it.

\subsection{Notation and Conventions} 
The most important convention of the paper is that everything is rationalized without comment. In particular all spectra are meant localized at $H\Q$ and all the homology and cohomology is meant with $\Q$ coefficients. Tensor products $\otimes$ are meant over $\Q$, or over the graded ring with only $\Q$ in degree zero and zero elsewhere. By subgroup of a compact Lie group we always mean closed subgroup. With the symbol $\T^r$ we mean the torus of rank $r$: compact connected abelian Lie group of dimension $r$. We denote $G=\T^2$ the $2$-torus in all the paper (except in the Appendix \ref{appendix:algmodel} where the theory is presented for $G$ a general torus), since is our fixed group of equivariance. The collection of connected closed codimension $1$ subgroups of $G$ is $\{H_i\}_{i\geq 1}$ indexed with $i\geq 1$, and with $H_1=1\times \T$ and $H_2=\T\times 1$ being the two privileged subgroups. We denote $H_i^j$ the subgroup with $j$ connected components and identity component $H_i$: we will refer to the subgroups with identity component $H_i$ as being along the $i$-direction. We denote with $F$ a generic finite subgroup of $G$ and with $H$ and $K$ generic closed subgroups of $G$ of any dimension. Given a module $M$ we will denote $\overline M$ the 2-periodic version of $M$: it is a graded module with $M$ in each even degree and zero in odd degrees. We denote elements in direct sums and products in the following way: $x=\{x_i\}_{i}\in \bigoplus_{i\geq 1}M_i$: this identifies the element $x$ in the direct sum that has $i$th component $x_i\in M_i$.

We will freely use the standard notation of schemes from Algebraic Geometry. We denote $\K(\X)$ the ring of meromorphic function for the algebraic variety $\X$, and $\eta(C)$ the generic point of a closed set $C$. We denote $\c$ our fixed elliptic curve over $\C$, $e$ is the identity of the elliptic curve and for a positive integer $n$: $\c[n]$ is the subgroup of elements of $n$-torsion, while $\c\gen n$ is the subset of elements of exact order $n$. We will use $P$ to denote a point of $\c$ of finite order. 

We will freely use the standard notation for algebraic models , and we recall it in the appendix \ref{appendix:algmodel}. In particular $\A(G)$ is an abelian category with graded objects and no differentials, while $d\A(G)$ is the category of objects of $\A(G)$ with differentials. Cohomology is unreduced unless indicated to the contrary with a tilde, so that $\hbg H=\tilde H^*(BG/H_+)$ is the unreduced cohomology ring. To ease the notation sometimes we will omit the base ring we are taking the tensor product over and denote it with an index: $\tens i$. This in turn means that we are considering the tensor product over the ring $\oefh i$ or its $F$th component $\hbg {H_i^{n_i}}$. We will make extensive use of the isomorphisms $\hbg {H_i^j}\cong \Q[c_{ij}]$ \eqref{eq:cij} and $\hbg F\cong \Q[\xa,\xb]$ \eqref{eq:xaxb}.

\subsection{Acknowledgements} This work is part of my PhD thesis at the University of Warwick under the supervision of John Greenlees. I am extremely grateful to John for his guidance, comments and ideas on this project. I would like in particular to thank Julian Lawrence Demeio for many useful conversations, suggestions and ideas especially on the Algebraic Geometry side. I am also grateful to Jordan Williamson and Emanuele Dotto for their helpful comments on the preliminary version of this paper.

\section{The correspondence subgroups-subvarieties}
\label{sec:1}
The goal of this section is to specify a correspondence between subgroups $H$ of $G=\T^2$ and certain varieties $\bvar H$. We have fixed an elliptic curve $\c$ over the complex numbers which defines a functor $\XF$ from compact abelian Lie groups to complex manifolds.
\begin{definition}
	\label{def:varfunctor}
	If $H$ is a compact abelian Lie group and $\c$ our fixed elliptic curve, define
	\begin{equation}
		\label{eq:functor}
		\XF(H):=\Hom_{\text{Ab}}(H^*,\c).
	\end{equation}
	Where we are considering group homomorphisms, and $H^*:= \Hom_{\text{Lie}}(H, \T)$ is the character group of $H$.
\end{definition}

Let $\X:=\XF(G)$ be the complex abelian surface defined by the 2-torus. The functor $\XF$ is exact and induces an embedding $\var H \hookrightarrow \X$ for every subgroup $H$ of $G$. Moreover $\var H$ has the same dimension as $H$ and is a subgroup of $\X$. We will only be interested in the functor \eqref{eq:functor} on subgroups of $G$, and therefore all the varieties $\XF(H)$ will be subvarieties of $\X$.

\begin{definition}
	\label{def:bvar}
	For every subgroup $H$ of $G$ define
	\begin{equation}
		\label{eq:xvar}
		\bvar H:=\var H \setminus \bigcup_K \var K 
	\end{equation}
	where the union is over all the proper subgroups $K$ of $H$ of finite index in $H$.
\end{definition}

Since the union in \eqref{eq:xvar} is finite, $\bvar H$ is a subvariety of $\X$, that for $G$ itself coincide with the all surface $\X$.
\begin{lemma}
	\label{lem:properties}
	For $H$ and $K$ subgroups of $G$ the following properties are satisfied:
	\begin{itemize}
		\item $\var {H\times K}=\var H \times \var K$.
		\item $\var {H\cap K}=\var H \cap \var K$
	\end{itemize}
\end{lemma}
\begin{proof}
	The first property follows immediately applying the functor $\XF$ to the exact sequence
	\begin{equation*}
		\begin{tikzcd}
			H  \arrow[r, tail] & H\times K \arrow[r, twoheadrightarrow] & K.
		\end{tikzcd}
	\end{equation*}
	
	For the second one we only need to prove the containment $\var H \cap \var K\subseteq \var {H\cap K}$, since the other containment is immediate from $\XF$ being a functor. Apply the exact functor $\XF$ to the commutative diagram:
	\begin{equation*}
		\begin{tikzcd}
			H\cap K \arrow[r, tail] & G \arrow[r, twoheadrightarrow] \arrow[dr, "p_0\times p_1"'] & \quotient G{H\cap K} \arrow[d, tail]\\
			& & \quotient GH \times \quotient GK.
		\end{tikzcd}
	\end{equation*}
	In doing so the right vertical maps remains injective, and therefore the kernel of $\var {p_0\times p_1}$ is $\var {H\cap K}$. Now it is enough to notice that every element in $\var H \cap \var K$ is sent to zero by $\var {p_0\times p_1}$. 
\end{proof}
\begin{rem}
	Applying this Lemma to $H=1\times \T$ and $K=\T\times 1$, we obtain $\X=\XF(G)= \XF(H)\times \XF(K)= \c \times \c$
\end{rem}

\subsection{The codimension 1 case}
Let $\{H_i\}_{i\geq 1}$ be the collection of connected codimension $1$ subgroups of $G=\T^2$, with $H_1=1 \times \T$ and $H_2=\T\times 1$. Each one of the $H_i$ can be written as the kernel of a nonzero character $z_i:G \to \T$ of $G$:
\begin{equation}
	\label{eq:char}
	\begin{tikzcd}
	H_i \arrow[r,tail] & G \arrow[r,twoheadrightarrow, "z_i"] & \T.
	\end{tikzcd}
\end{equation}

Moreover we may choose  $z_i=z_1^{\lambda_i}z_2^{\mu_i}$ for a couple of coprime integers $(\lambda_i, \mu_i)$ not both zero and with $\mu_i\geq 0$.
Applying the functor $\XF$ to \eqref{eq:char}, the  subvariety $\var {H_i}$ can be described in the same way as the kernel of the projection $\pi_i:=\var {z_i}$:
\begin{equation}
	\label{eq:pii}
	\begin{tikzcd}
		\var {H_i} \arrow[r,tail]& \X \arrow[r,twoheadrightarrow, "\pi_i"] & \c.
	\end{tikzcd}
\end{equation}
Where the relation $\pi_i=\lambda_i\pi_1+\mu_i\pi_2$ holds now by the group law of the elliptic curve.
\begin{definition}
	\label{def:zij}
	For every $i\geq 1$ and $j\in \Z\setminus \{0\}$ we define the character $z_i^j$ of $G$ post-composing $z_i$ with the $j$-th power map of $\T$. We also define $\pi_i^j:=\var {z_i^j}$, note that this map is obtained post-composing $\pi_i$ with the $j$-th power map in $\c$.
\end{definition}

\begin{definition}
	For every direction $i\geq 1$ and every $j\geq 1$ define the $(i,j)$-divisor:
	\begin{equation}
		\label{eq:dij}
		D_{ij}:=\bvar {H_i^j}.
	\end{equation}
	Where $H_i^j$ is the subgroup of $G$ with $j$ connected components and identity component $H_i$.
\end{definition}
\begin{definition}
	For $P\in \c$ a point of finite order define:
	\[
	D_{i,P}:=\pi_i^{-1}(P)
	\]
\end{definition}
\begin{rem}
    \label{rem:decomposition}
	Notice that $\var {H_i^j}=\pi_i^{-1}(\c[j])$, while $\dij=\pi_i^{-1}(\c\gen j)$. Therefore we have the decompositions:
	\begin{equation}
		\label{eq:decompositions}
		\begin{split}
			\bvar {H_i^j}= \dij &= \coprod_{P\in \c\gen j}D_{i,P}\\
			\var {H_i^n} &= \coprod_{j \mid n } \dij.
		\end{split}
	\end{equation}
\end{rem}
\begin{rem}
	From now on we will always refer to $\pi_i$ as the projection along the $i$-direction. All the varieties $\var {H_i^j}$, $\dij$ and $\dip$ will all be referred as ``along the $i$-direction''. We denote $D_i=D_{i,1}=D_{i,e}=\var {H_i}$. Note from \eqref{eq:decompositions} that the subvarieties along the $i$-direction $D_{ij}$ are all parallel, disjoint and made up of disjoint pieces $D_{i,P}$ isomorphic to a single copy of $\c$.
\end{rem}
\subsection{The codimension 2 case} If $F$ is a finite subgroup of $G$, then $\var F$ is a finite collection of closed points of $\X$. The subset $\bvar F\subset \var F$ satisfies some desirable properties. 
\begin{lemma}
	\label{lem:disj}
	If $F\neq F'$ are finite subgroups of $G$, then $\bvar F\cap \bvar {F'}=\emptyset$
\end{lemma}
\begin{proof}
	Suppose $Q\in \bvar F \cap \bvar {F'}$, then $Q\in \var F \cap \var {F'}=\var {F\cap F'}$. Without loss of generality $F\cap F'$ is a proper subgroup of $F$, and therefore $Q\notin \bvar F$. 
\end{proof}
\begin{lemma}
	\label{lem:fin}
	Given $F<G$ finite, then for every direction $i\geq 1$ there exists one and only one index $n_i=n_i(F)\geq 1$ such that
	\[
	\div i\cap \bvar F\neq \emptyset,
	\]
	precisely the only index $n_i$ such that $H_i^{n_i}$ is the subgroup generated by $H_i$ and $F$ (Definition \ref{def:ni}). Moreover
	\begin{equation}
		\label{eq:intersect}
		\bvar F=\bigcap_{i\geq 1} \div i
	\end{equation}
\end{lemma}
\begin{proof}
	Given a finite subgroup $F$, for every direction $i\geq 1$ let $n_i$ be the integer such that $H_i^{n_i}=\gen {H_i, F}$. Since $F\subseteq H_i^{n_i}$ then $\var F\subseteq \var {H_i^{n_i}}$. Recall from \eqref{eq:decompositions} the decomposition:
	\begin{equation}
		\label{eq:dechini}
		\var {H_i^{n_i}}=\coprod_{j\mid n_i}\dij.
	\end{equation}

	We start by proving $D_{i,n}\cap \bvar F=\emptyset$ for $n\neq n_i$:
	\begin{itemize}
		\item If $n\nmid n_i$ then $D_{i,n}$ is disjoint from $\bvar F$ since from \eqref{eq:dechini} it is disjoint from $\var {H_i^{n_i}}$.
		\item If $n\mid n_i$ but $n\neq n_i$ then $H_i^n\subsetneq H_i^{n_i}$ and $F\nsubseteq H_i^{n}$ since $n_i$ is the minimum integer for which the containment is true. Therefore $F':=H_i^n\cap F$ is a proper subgroup of $F$, and as such:
		\[
		D_{i,n}\cap \var F\subseteq \var {F'}
		\]
		which implies $D_{i,n}\cap \bvar F=\emptyset$.
	\end{itemize}

	To prove $D_{i,n_i}\cap \bvar F\neq \emptyset$ and the second part of the statement we use again the decomposition \eqref{eq:dechini}. Since $\bvar F$ can intersect only $\div i$ and it is contained in $\var {H_i^{n_i}}$, it must be contained in $\div i$. It also follows $\bvar F$ is contained in the intersection \eqref{eq:intersect}. 
	
	We are left to prove that if $Q\in \bigcap_{i\geq 1}\div i$, then $Q\in \bvar F$. First, for every direction $i\geq 1$, $Q\in \XF(H_i^{n_i})$, therefore 
	\[
	Q\in \bigcap_{i\geq 1}\var {H_i^{n_i}}=\var {\bigcap_{i\geq 1}H_i^{n_i}}=\var F.
	\] 
	Now suppose $Q\in \XF(F')$ for a proper subgroup $F'$ of $F$. For every $i\geq 1$ define $n_i'$ such that $H_i^{n_i'}=\gen {H_i,F'}$. Then 
	\begin{equation*}
		\bigcap_{i\geq 1} H_i^{n_i} = F \neq F' = \bigcap_{i\geq 1} H_i^{n_i'}.
	\end{equation*}
	 Therefore it exists an index $s$ for which $n_s'\neq n_s$, but then $Q\in D_{s,n_s}$ and $Q\in D_{s,n_s'}$ which is absurd since they are disjoint. In conclusion $Q\in \bvar F$.
\end{proof}

\begin{lemma}
	For every finite subgroup $F\leq G$, the subset $\bvar F$ is non empty.
\end{lemma}
\begin{proof}
	If $F=\{1\}\times \{1\}$ is the trivial subgroup we have $\bvar {\{1\}\times \{1\}}=\var {\{1\}\times \{1\}}=\{e\}\times \{e\}\neq \emptyset$.
	
	If $F= \Z_{p^n}\times \Z_{p^m}$ is a $p$-group, with $p$ prime and $n\leq m$, then applying Lemma \ref{lem:properties}: 
	\[
	\var F\cong \var {\Z_{p^n}} \times \var {\Z_{p^m}}\cong \c[p^n]\times \c[p^m].
	\]
	Every proper finite subgroup of $F$ is contained in a maximal one (i.e. proper subgroup not contained into any other proper subgroup), therefore we can simply consider the maximal subgroups of $F$. 
	If $F$ is cyclic ($n=0$ and $m>0$) the only maximal subgroup of $\Z_{p^m}$ is $\Z_{p^{m-1}}$ and since $\var {\Z_{p^m}}\cong \c[p^m]$ has $p^{2m}$ points while $\var {\Z_{p^{m-1}}}$ has only $p^{2m-2}$ we have 
	\[
	\bvar {\Z_{p^m}}=\var {\Z_{p^m}}-\var {\Z_{p^{m-1}}}\neq \emptyset.		
	\]
	If $F$ is not cyclic, we need to use the following computations and facts. Every maximal subgroup $F'$ of $F$ has index $p$, therefore $\var {F'}$ has $p^{2n+2m-2}$ points. There are exactly $p+1$ maximal subgroups in $F$. Therefore there are at most $(p+1)p^{2m+2n-2}$ points contained in $\bigcup_{F'}\var {F'}$ where $F'$ ranges on all maximal subgroups of $F$. Since $\var F$ has $p^{2m+2n}$ points and $p+1<p^2$ it follows $\bvar F\neq \emptyset$.
	
	For the general case of a finite subgroup $F$, decompose it into a product of $p$-groups:
	\[
	F\cong F_{p_1}\times \dots \times F_{p_k}.
	\]
	Applying the previous case we can pick for each prime $p_i$ a point $Q_i\in \bvar {F_{p_i}}$. The point
	\[
	(Q_1,\dots, Q_k)\in \var F\cong \var {F_{p_1}}\times \dots \times \var {F_{p_k}}
	\]
	is a point in $\bvar F$. Indeed any maximal subgroup $F'$ of $F$ is of the following form: pick one of the factors $1\leq i\leq k$, and a maximal subgroup $F'_{p_i}<F_{p_i}$, replace $F_{p_i}$ with $F_{p_i}'$ in the product 
	\[
	F'=F_{p_1}\times \dots \times F'_{p_i} \times \dots \times F_{p_k}.
	\]
	The point $(Q_1,\dots, Q_k)$ cannot be in any of these $\var {F'}$ since $Q_i\notin \var {F'_{p_i}}$ for each $i$.
\end{proof}

The bottom line is that we have associated to subgroups of $G$, certain subvarieties of $\X$ of the same dimension:
\begin{center}
	\begin{tabular}{ccc}
		\toprule
		Codimension & Subgroups of $G$ & subvarieties of $\X$ \\
		\midrule
		$0$ & $G$ & $\X$ \\
		$1$ & $H_i^j$ & $\dij$ \\
		$2$ & $F$ & $\bvar F$ \\
		\bottomrule
	\end{tabular}
\end{center}

\section{Change of topology}
\label{sec:2}
In this section we change the topology on the algebraic variety $\zar \X=\c\times \c$. We use $\zar \X$ to denote the algebraic variety with the usual Zariski topology. Recall $G=\T^2$ and that we have fixed an elliptic curve $\c$ over $\C$. 
\begin{definition}
	\label{def:tp}
	Over the set $\zar \X$ define the torsion point topology $\tp \X$ with generating closed subsets $\{\dij\}_{ij}$, where $i\geq 1$ and $j\geq 1$ (recall the definition of $\dij$ in \eqref{eq:dij}).
\end{definition}
\begin{rem}
	\label{rem:sober}
	A delicate remark is imperative here. The topological space $\tp \X$ is not $T_0$: there are different points that are topologically undistinguishable (i.e. they are exactly in the same open subsets for the TP-topology). In view of the previous section and in particular Lemma \ref{lem:fin}, we have a straightforward way to check when two points in $\tp \X$ are topologically undistinguishable: they are topologically undistinguishable if and only if they belong exactly to the same subsets in the collection $\{\dij\}_{ij}$. For example the generic point of $\zar \X$ is topologically undistinguishable from any point that does not belong to any of the $\dij$. This will not be an issue until we will need the theory of Cousin complexes as developed by Grothendieck. To apply all that machinery we will need a sober space (i.e. every closed irreducible subset has a unique generic point), and $\tp \X$ is not a sober space. There is an easy fix to this which is to consider the Kolmogorov quotient of $\tp \X$: $\KQ(\tp \X)$, namely the quotient of $\tp \X$ by the equivalence relation of topologically undistinguishable points. 
	For the sake of clarity we will not change the underlying set of $\zar \X$ and $\tp \X$ and we ask the reader to keep in mind that all the results we will prove for $\tp \X$ in this section apply as well if you substitute $\tp \X$ with its Kolmogorov quotient, since topologically undistinguishable points have exactly the same stalks and nothing changes from the point of view of sheaves. The same is true with the results in the rest of this paper with the only exceptions of the theory of Cousin complexes where we will stress the need of a sober space.
\end{rem}

Notice that every subset in the generating collection $\{\dij\}_{ij}$ is a closed subset also in the Zariski topology. Therefore every TP-open is also a Zariski-open, and we have a well defined continuous map $\phi: \zar \X \to \tp \X$. 

\subsection{The pushforward is exact on quasi-coherent sheaves}
Since $\phi$ is continuous it induces a pair of adjoint functors \cite[Section II.4]{iversen:coho} called pushforward and pullback (or direct and inverse image) of sheaves:
\begin{equation*}
	\begin{tikzcd}
	\Shv(\tp \X)\ar[r,bend left,"\phi^{-1}",""{name=A, below}] & 
	\Shv(\zar \X) \ar[l,bend left,"\phi_*",""{name=B,above}] \ar[from=A, to=B, symbol=\dashv]
	\end{tikzcd}
\end{equation*}

For a Zariski-sheaf $\F$ its pushforward sheaf is defined on a TP-open $U$ by:
\[
(\phi_*\F)(U):=\F(\phi^{-1}(U))
\]
This defines indeed a TP-sheaf, moreover $\phi_*$ is in general always left exact, and sends injective sheaves to injective sheaves since its left adjoint is exact \cite[Corollary 4.13]{iversen:coho}. The functor $\phi^{-1}$ is left adjoint to $\phi_*$ \cite[Theorem 4.8]{iversen:coho}, but it is not true in general that the pushforward is exact. However in our situation $\phi_*$ is exact on quasi-coherent sheaves since in the TP topology we still have an open cover of Zariski affines:
\begin{lemma}
	\label{lem:affine}
	Every point in $\tp \X$ is contained in a TP-open which is an open affine in the Zariski topology.
\end{lemma}
\begin{proof}
	In $\zar \X$ the complement of the union of two TP-closed subsets $D_{ij}$ and $D_{rs}$ along different directions ($i\neq r$) is a TP-open which is affine in the Zariski topology. 
\end{proof}
\begin{corollary}
	\label{cor:exact}
	The functor $\phi_*$ is exact on quasi-coherent sheaves.
\end{corollary}
\begin{proof}
	Consider a point $x\in \tp \X$, applying Lemma \ref{lem:affine} we can compute the TP-stalk at $x$ as a colimit over TP-opens that are open affines in the Zariski topology. By \cite[Theorem 3.5]{hartshorne} taking sections over an open affine $\Gamma(\Spec(R),\F)$ for a quasi-coherent Zariski sheaf $\F$ is an exact functor. Therefore the functor $\F \to (\phi_*(\F))_x$ is  exact since it is a colimit of exact functors.   
\end{proof}
\begin{corollary}
	\label{cor:coh}
	If $\F$ is a quasi-coherent Zariski sheaf then its cohomology in both topologies is the same:
	\begin{equation}
		\label{eq:commcoh}
		H^*(\zar \X, \F)\cong H^*(\tp \X, \phi_*(\F))
	\end{equation}
\end{corollary}
\begin{proof}
	The functor $\phi_*$ is exact on quasi-coherent sheaves by Corollary \ref{cor:exact} and preserves injective sheaves \cite[Corollary 4.13]{iversen:coho}. Therefore we can consider a quasi-coherent injective resolution of $\F$ \cite[\href{https://stacks.math.columbia.edu/tag/077K}{Tag 077K}]{stacks-project}. Applying $\phi_*$ to this injective resolution we have an injective resolution of $\phi_*(\F)$. It is enough now to notice that taking global section for a Zariski sheaf or its pushforward is the same.
\end{proof}
\begin{nota}
    We will denote by $H^*(\X,\F)$ the common value of these two cohomologies, and denote by
\[
\tp \ox:=\phi_*(\zar \ox)
\] 
the pushforward of the Zariski structure sheaf. In the rest of the paper we use this TP notation to indicate the pushforward $\phi_*$.
\end{nota}
 The space $(\tp \X, \tp \ox)$ is a ringed topological space and the functor $\phi_*$ factors through the respective categories of modules:
\[
\phi_*:\Mod(\zar \ox) \to \Mod(\tp \ox).
\]

We are interested in explicitly computing the TP-topology stalks of the pushforward of a Zariski-sheaf $\F$.
\begin{itemize}
	\item If $x=\eta(\tp \X)$ is a generic point of the whole space, we can take a colimit of complements of increasingly bigger unions of the $\dij$:
	\begin{equation}
		\label{eq:stalkgen}
		(\phi_*\F)_{x}=\varinjlim_{n\to \infty}\F(\zar \X \setminus \bigcup_{i,j\leq n} \dij).
	\end{equation}
	For $\F=\zar \ox$ the colimit above picks the regular functions on the complement of increasingly bigger unions of the $\dij$. This yields those meromorphic functions on $\zar \X$ that are allowed poles only in the collection $\{\dij\}$.
	\begin{equation}
		\label{eq:k}
		\K:= \tp {\oxx {x}}=\{f\in \K(\zar \X)\mid f\,  \text{is allowed poles only at}\, \{\dij\}\}.
	\end{equation}
	\item If $x=\eta(\dij)$ is a generic point of a generating closed subset, simply skip $\dij$ itself in \eqref{eq:stalkgen}. For $\F=\zar \ox$ this yields:
	\begin{equation}
		\label{eq:odij}
		\O_{\dij}:= \tp {\oxx x}=\{f\in \K\mid f \,\text{is regular at} \, \dij\}.
	\end{equation}
	We also denote $m_{ij}<\odij$ the ideal of those functions vanishing at $\dij$.
	\item If $x \in \bvar F$
	(which automatically makes it also a generic point for $\bvar F$), then in \eqref{eq:stalkgen} simply skip all the $\dij$ containing $\bvar F$. By Lemma \ref{lem:fin} for every direction $i\geq 1$, only $D_{i,n_i}$ contains $\bvar F$. Therefore:
	\begin{equation}
		\label{eq:of}
		\O_F:= \tp {\oxx x}=\{f\in \K\mid \, \forall i\geq 1, \, f \, \text{is regular at}\, D_{i,n_i} \}.
	\end{equation}
	We also denote $m_F<\O_F$ the ideal of those functions vanishing at $\bvar F$.
\end{itemize}
\begin{rem}
	\label{rem:stalk}
	When $\F$ is a quasi-coherent $\zar \ox$-module we can use commutative algebra to compute the stalk at a point $x\in \tp \X$. Pick a TP-open containing the point which is an open affine in the Zariski topology (Lemma \ref{lem:affine}): $U=\Spec(R)$. Then $\F$ restricted to that open is isomorphic to the sheaf $\widetilde M$ for an $R$-module $M$. Modulo restricting the affine open $U$, deleting the closed subset $\dij$ from $U$ corresponds to inverting those elements in $R$ that vanish at $\dij$. Therefore the stalk at $x$ in the TP-topology is $S^{-1}M$ for the multiplicatively closed subset $S$ generated by those elements vanishing at TP-closed subsets. Notice that it is exactly as in the Zariski topology, with the only difference that instead of inverting everything outside that prime, we invert just those elements outside that prime corresponding to the generating closed subsets for the TP-topology. 
\end{rem}

\subsection{Choice of coordinates}
The aim of this subsection is to build a set of uniformizers for the subvarieties $\dij$ with respect to the TP-topology. We construct them over the algebraic variety $\X$ with its normal Zariski topology, but notice that they are defined for the TP-topology as well (they belong in $\K$). All the uniformizers we build here only depend upon a choice of a coordinate $t_e\in \O_{\c,e}$ vanishing to the first order at $e$, and with poles only at points of finite order of $\c$.
 
We begin by recalling the definition of the TP-topology for the single elliptic curve \cite[Definition 7.1]{john:elliptic}:

\begin{definition}
    \label{def:tptopc}
	Over the set $\c$ define the torsion point topology $\tp \c$ with generating closed subsets $\{\c\gen n\}_{n\geq 1}$, where $\c\gen n$ are the elements of exact order $n$ in $\c$.
\end{definition}

This is a ringed topological space with the pushforward of the structure sheaf: $\tp \O_{\c}$. Likewise \cite[Definition 8.2]{john:elliptic} choose a coordinate for $\c$ at $e$:
\begin{choice}
	\label{choice:te}
	Choose $t_e\in \tp \O_{\c,e}\subset \O_{\c,e}$ vanishing to the first order at $e$.
\end{choice}

\begin{rem}
	We denote $m_e<\O_{\c,e}$ the maximal ideal of those functions vanishing at $e$. Then $m_e$ is principal with generator $t_e$, and the same is true if we restrict $m_e$ to  $\tp \O_{\c,e}$. Note that by \cite[Proposition 8.1]{hochster:local} we obtain the same result if we complete those two two rings with respect to those two maximal ideals:
	\begin{equation}
		\label{eq:completed}
		(\tp \O_{\c,e})^{\wedge}_{m_e}\cong ( \O_{\c,e})^{\wedge}_{m_e}\cong \C[[t_e]].
	\end{equation}
	Elements in \eqref{eq:completed} are functions defined in a formal neighbourhood of the identity of $\c$, and can be written as formal power series with complex coefficients in the variable $t_e$.
\end{rem}

The isomorphism $\X = \c\times \c$ is given through the two projections $\pi_1:\X \to \c$ and $\pi_2:\X \to \c$. The pullbacks of the coordinate: $t_1:=\pi_1^*(t_e)$ and $t_2:=\pi_2^*(t_e)$ respectively define uniformizers for $D_1=\{e\}\times \c$ and $D_2=\c\times \{e\}$ and together they generate the maximal ideal $m$ in the stalk $\O_{\X,O}$ of those functions that vanishes at $O=(e,e)$. Exactly as before when we complete with respect to $m$ we get $(\O_{\X,O})^{\wedge}_m\cong \C[[t_1,t_2]]$.

This is a way to manifest the formal group law of the elliptic curve $\c$. If $g:\X= \c\times \c \to \c$ is the group law of the elliptic curve, we have an induced map on the completed rings
\begin{equation*}
	\begin{split}
		g^*:\C[[t_e]]\cong (\O_{\c,e})^{\wedge}_{m_e} &\to (\O_{\X,O})^{\wedge}_{m} \cong \C[[t_1,t_2]]\\
		t_e &\mapsto F(t_1,t_2)
	\end{split}
\end{equation*}
The element $F(t_1,t_2)$ is the formal group law of the elliptic curve $\c$ with respect to the uniformizer $t_e$. Since we are over a field of characteristic zero there exists a unique logarithm for $F$ (See for example \cite[Proposition 3.1]{strickland:formal}), namely a strict isomorphism with the additive formal group law:
\begin{theorem}
	There exists a unique element $\hat t_e\in (\O_{\c,e})^{\wedge}_{m_e}$ that can be written as a formal power series with complex coefficients:
	\begin{equation}
		\label{eq:powerserieste}
		\hat t_e:=f(t_e)=\sum_{k=1}^{\infty}\alpha_k t_e^k\in \C[[t_e]]
	\end{equation}
	with $\alpha_1=1$ and such that
	\[
	f(F(t_1,t_2))=f(t_1)+f(t_2).
	\]
\end{theorem}
As an immediate corollary:
\begin{corollary}
	\label{lem:lincomb}
	Given two integers $r,s \in \Z$ the linear map
	\begin{equation}
		\label{eq:linearmaps}
		\begin{split}
			\X\cong \c\times \c &\xrightarrow{(r,s)} \c\\
			(x,y) &\mapsto rx+sy
		\end{split}
	\end{equation}
	induces on the completed local rings a map:
	\[
	(r,s)^*:(\O_{\c,e})^{\wedge}_{m_e} \to (\O_{\X,O})^{\wedge}_{m}
	\]
	such that
	\[
	(r,s)^*(f(t_e))=rf(t_1)+sf(t_2).
	\]
\end{corollary}

We can now simply pullback $t_e$ and $\hat t_e$ along the various projections $\pi_i^j:\X\to \c$ (Definition \ref{def:zij}). Note that $\pi_i^j=(j\lambda_i, j\mu_i)$ is a linear map of the kind of \eqref{eq:linearmaps}, and therefore it induces maps on the completed and uncompleted local rings at the identities.
\begin{definition}
	For every $i\geq 1$ and $j\geq 1$ define the coordinate
	\begin{equation}
		\label{eq:tij}
		t_{ij}:=(\pi_i^j)^*(t_e)\in \O_{\X,O}
	\end{equation}
\end{definition}
	\begin{rem}
		Since $t_e\in \tp \O_{\c,e}$ we have that $\tij\in \K$, namely it has poles only in the collection of generating closed for the TP-topology. Moreover $\tij$ vanishes at first order at $\var {H_i^j}$ and therefore at $\dij$. This yields $\tij \in \O_{\dij}$ (defined in \eqref{eq:odij}), and that it generates the principal ideal $m_{ij}$ of those functions vanishing at $\dij$. 
	\end{rem}
	\begin{definition}
	\label{def:coord}
		For every $i\geq 1$ and $j\geq 1$ define the completed coordinate
		\begin{equation}
			\label{eq:hatij}
			\hat t_{ij}:=(\pi_i^j)^*(\hat t_e) \in (\O_{\X,O})^{\wedge}_{m}=(\tp {\O_{\X,O}})^{\wedge}_{m}
		\end{equation}
	\end{definition}
	\begin{rem}
	Note that $(\pi_i^j)^*:(\O_{\c,e})_{m_e}^{\wedge}\to (\O_{\X,O})_m^{\wedge}$ is a continuous map of completed rings over $\C$, therefore $\hatij$ can be expressed using the power series \eqref{eq:powerserieste} in the variable $\tij$:
	\begin{equation}
		\label{eq:hatijpowerseries}
		\hat t_{ij}=(\pi_i^j)^*(\sum_{k=1}^{\infty}\alpha_k t_e^k)=\sum_{k=1}^{\infty}\alpha_k t_{ij}^k.
	\end{equation}
	From this expansion it is transparent that $\hat t_{ij}\in (\O_{\dij})^{\wedge}_{m_{ij}}$.
	\end{rem}

\section{Cousin complex}
\label{sec:3}
The aim of this section is to prove that the Cousin complex of the structure sheaf $\O=\tp \ox$ for the TP-topology is a flabby resolution of $\O$. Here Cousin complex is meant in the sense of the theory developed in \cite[Chapter 4]{hartshorne:residues} (all the quotes refer to that chapter). To apply in full the machinery developed in that chapter we need a sober topological space, and $\tp \X$ is not sober as discussed in Remark \ref{rem:sober}. To fix this issue we substitute $\tp \X$ with its Kolmogorov quotient to obtain a sober space. We note as explained in Remark \ref{rem:sober} that all the other results including the ones in the previous section apply indifferently to $\tp \X$ and its Kolmogorov quotient. We do not change notation for it, and we highlight that being sober is necessary in particular for the filtration \eqref{eq:filt}, and for the splitting \eqref{eq:ccdec}, where we need to index on the generic points of the irreducible closed subsets. Everything we say about stalks, sheaves, sections and supports apply unchanged to $\tp \X$ and its Kolmogorov quotient.  

To apply in full the machinery of Cousin complexes we need a topological space $X$ and a filtration $X=Z^0\supseteq Z^1 \supseteq \dots$ satisfying the following hypothesis:
\begin{hyp}
	\label{hyp:cousin}
	\begin{enumerate}
		\item $X$ is locally Noetherian.
		\item $X$ is sober: every closed irreducible subset of $X$ has a unique generic point.
		\item The filtration is stable under specialization.
		\item Every element in $Z^n\setminus Z^{n+1}$ is maximal under specialization  for every $n\geq 0$.
		\item The filtration is separated: $\bigcap_{n\geq 0}Z^n=\emptyset$.
	\end{enumerate}
\end{hyp}
The only filtration we will consider is the codimension filtration of a topological space $X$: 
\begin{definition}
	Define the codimension filtration $X=Z^0\supseteq Z^1 \supseteq \dots$ of a topological space $X$ as:
	\begin{equation}
		\label{eq:filt}
		Z^n:=\{x \in X\mid \Codim(x)\geq n\}.
	\end{equation}
\end{definition}
\begin{proposition}
	\label{prop:cousincomplex}
	Since the Kolmogorov quotient $\tp \X$ with the codimension filtration satisfies the Hypothesis \ref{hyp:cousin} then by \cite[Proposition 2.3 and 2.5]{hartshorne:residues} we can consider the Cousin complex of $\O=\tp \ox$:
	\begin{equation}
		\label{eq:cc1}
		\begin{tikzcd}
			\O \arrow[r] & \CC^0(\O) \arrow[r, "d_0"] & \CC^1(\O) \arrow[r, "d_1"] & \CC^2(\O) \arrow[r] &0.
		\end{tikzcd}
	\end{equation}
	Each term can be decomposed further \cite[Variation 8 Motif F pg. 225]{hartshorne:residues}:
	\begin{equation}
		\label{eq:ccdec}
		\CC^n(\O):=\underline{H}^n_{Z^n/Z^{n+1}}(\O)= \bigoplus_{x\in Z^n\setminus Z^{n+1}}\iota_x(\h n x(\O)).
	\end{equation}
\end{proposition}

\begin{nota}
	Few words on the notation and definitions of Grothendieck. An underlined $\underline{H}$ is to signal that we are considering the cohomology sheaf. For an abelian group $M$, the sheaf $\iota_x(M)$ denotes the constant sheaf with value $M$ on the closure of the point $x$ (or the constant sheaf on the closed subset $D$ in case of $\iota_D(M)$). The cohomology $\h n x(\O)$ is a punctual invariant \cite[Variation 8]{hartshorne:residues}:
\end{nota}

\begin{definition}
	\label{rem:stalksupp}
	Fixed a point $x$ in a topological space, with closure $Z=\overline{\{x\}}$, define the functor $\F \mapsto \Gamma_x(\F)$, that associates to a sheaf $\F$ the subgroup of the stalk $\F_x$ of those germs that admits a representative supported in $Z$:
	\begin{equation}
		\label{eq:gammax}
		\Gamma_x(\F):=\{\alpha\in \F_x\mid \exists(s,U), [(s,U)]=\alpha, \supp(s)\subseteq Z \}.
	\end{equation}
	Then $\h nx(\_)$ is defined to be the $n$-th right derived functor of $\Gamma_x(\_)$. Recall that the support of a section $s$, $\supp(s)$, consists of those points for which the section is non-zero in the stalk: $s_x\neq 0$. 
\end{definition}
\begin{nota}
	To ease the notation when the point $x$ is the generic point of a closed subset $Z$ we will denote $\h nZ(\F):=\h n{\eta(Z)}(\F)$.
\end{nota}
\begin{rem}
	\label{rem:loccoh}
	if $Z=\overline{\{x\}}$ then by \cite[Variation 8]{hartshorne:residues}:
	\begin{equation}
		\label{eq:stalkiso}
		\h nx(\F)\cong (\underline H^n_Z(\F))_x.
	\end{equation}
\end{rem}
For the space $\tp \X$ notice $Z^3=\emptyset$ and
\begin{itemize}
	\item $Z^0 \setminus Z^1=\{\eta (\tp \X)\}$.
	\item $Z^1 \setminus Z^2=\{\eta (\dij)\}_{ij\geq 1}$.
	\item $Z^2 \setminus Z^3=\{\eta (\bvar F)\}_F$, where $F$ ranges over all finite subgroups of $G$.
\end{itemize}
Therefore we can write explicitly the Cousin complex of $\O=\tp \ox$ \eqref{eq:cc1} using the decomposition  \eqref{eq:ccdec}:
\begin{equation}
	\label{eq:cc2}
		\O \longrightarrow \iota_{\X}(\h 0{\X}(\O)) \xrightarrow{d_0} \bigoplus_{i,j\geq 1}\iota_{\dij} (\h 1{\dij}(\O)) \xrightarrow{d_1} \bigoplus_{F}\iota_F(\h 2F(\O)) \longrightarrow 0
\end{equation}
\begin{nota}
	We denote $\iota_F(\h 2 F(\O)):=\iota_{\bvar F}(\h 2{\bvar F}(\O))$.
\end{nota}

\subsection{The Cousin complex is a flabby resolution}
We want to show that the Cousin complex \eqref{eq:cc2} of $\O$ is a flabby resolution of $\O$. To do so we need first to discuss the support of sections in the TP-topology.

Let us start recalling the analogous definition of the support functor in commutative algebra: 
\begin{definition}
	\label{def:commsupp}
	Given a commutative ring $R$ and an ideal $I$ of $R$ for any $R$-module $M$, define the functor of elements supported in $I$:
	\begin{equation}
		\label{eq:gammaim}
		\Gamma_{I}(M):=\{s\in M\mid \exists n\geq 0, I^ns=0\}
	\end{equation}
	and its right derived functors $H_I^*(M)$, which are all $R$-modules.
\end{definition}

 This definition is analogous to the sheaf version of cohomology with support since these two cohomologies coincide on affine schemes. More precisely \cite[Theorem 2.3]{hartshorne:local}:
\begin{theorem}
	\label{lem:affineloc}
	Let $R$ be a Noetherian ring, $U=\Spec(R)$, $I$ a finitely generated ideal of $R$ with corresponding closed subset $V(I)$, and $M$ an $R$-module. Then
	\[
	H^*_{V(I)}(U,\widetilde M)\cong H^*_I(M).
	\]
\end{theorem}
This is because the two support functors identify the same submodule of $M$: $s\in \Gamma(\Spec(R), \widetilde M)$ is a section with support in $V(I)$ if and only if $\exists n\geq 0$ such that $I^ns=0$. 

\begin{lemma}
	\label{lem:changesupp}
	Let $Z$ be an irreducible TP-closed subset. Then in a TP-open subset $U=\Spec(R)$ which is affine in the Zariski topology, for any quasi-coherent $\zar \ox$-module $\F$ the sections in $\F(U)$ with support in $Z$ are the same in both topologies:
	\[
	\zar {\Gamma_Z}(U, \F)=\tp {\Gamma_Z}(U, \phi_*\F).
	\]
	Therefore their right derived functors are isomorphic:
	\[
	H^*_Z(U, \F)\cong H^*_Z(U, \phi_*\F)
	\]
\end{lemma}
\begin{proof}
	First let us prove the containment $\tp {\Gamma_Z}(U, \phi_*\F) \subseteq \zar {\Gamma_Z}(U, \F)$. If $s \in \F(U)$ is a section whose TP-support is contained in $Z$ then its Zariski-support is also contained in $Z$. If $x\in U\setminus Z$, by Remark \ref{rem:stalk}, modulo restricting $U$, the section $s$ is zero in the Zariski-stalk at $x$ since it is already zero in the TP-stalk at $x$ where less elements are inverted. 
	
	Let us prove the other containment $\zar {\Gamma_Z}(U, \F)\subseteq \tp {\Gamma_Z}(U, \phi_*\F)$. If $s\in \F(U)$ is a section whose Zariski-support is contained in $Z$, then by Lemma \ref{lem:affine} there exists $n\geq 0$ such that $\I(U)^ns=0$, where $\I$ is the ideal sheaf of $Z$. If $x\in U\setminus Z$, by Remark \ref{rem:stalk}, modulo restricting $U$, to obtain the TP-stalk at $x$ from $\F(U)$ we are inverting at least one element in $\I(U)$. If by absurd none of the elements in $\I(U)$ were inverted, then every $\dij$ containing $Z$ will also contain $x$, implying $x\in Z$ since $Z$ is irreducible in the TP-topology. In conclusion the section $s$ is zero in the TP-stalk at $x$ since $\I(U)^ns=0$ and at least one element of $\I(U)$ is inverted in the TP-stalk. 
	
	To show that their right derived functors are isomorphic pick an injective resolution of $\F$ in quasi-coherent $\zar \ox$-modules \cite[\href{https://stacks.math.columbia.edu/tag/077K}{Tag 077K}]{stacks-project}. The pushforward along $\phi$ of this injective resolution is an injective resolution of $\phi_*\F$ in $\tp \ox$-modules, since by Corollary \ref{cor:exact} $\phi_*$ is exact on quasi-coherent sheaves and preserves injectives.
\end{proof}
We recall \cite[Proposition 2.6]{hartshorne:residues} the equivalent definitions of Cohen-Macaulay sheaf with respect to a filtration $Z^n$ (We use as well the decomposition \eqref{eq:decompositions}):
\begin{proposition}
	\label{def:cohen}
	Under the Hypothesis \ref{hyp:cousin} for a sheaf of abelian groups $\F$ the following are equivalent:
	\begin{enumerate}
		\item $\underline{H}^i_{Z^n}(\F)=0$ for all $i\neq n$.
		\item $\underline{H}^i_{Z^n/Z^{n+1}}(\F)=\bigoplus_{x\in Z^n\setminus Z^{n+1}}\iota_x(\h i x(\F))=0$ for all $i\neq n$.
		\item The Cousin complex of $\F$ is a flabby resolution of $\F$.
	\end{enumerate}
	The sheaf $\F$ is said to be Cohen-Macaulay when it satisfies any of these equivalent conditions.
\end{proposition}
\begin{corollary}
	\label{cor:cohmac}
	If $\F$ is a quasi-coherent $\zar {\ox}$-module Cohen-Macaulay with respect to the codimension filtration in $\zar \X$, then $\phi_*\F$ is Cohen-Macaulay with respect to the codimension filtration in $\tp \X$.
\end{corollary}
\begin{proof}
	We prove that for the sheaf $\phi_*(\F)$ condition (2) of Proposition \ref{def:cohen} is satisfied. Since the Hypothesis \ref{hyp:cousin} are satisfied then by \cite[Lemma 2.4]{hartshorne:residues} we only need to prove condition (2) when $i<n$ since for $i>n$ is automatically satisfied. Therefore we need to show that for every $x\in \tp \X$ with closure $Z$ and $i<\Codim(x)$ we have:
	\begin{equation}
		\label{eq:hix}
		\h ix(\phi_*\F)\cong (\underline H^i_Z(\phi_*\F))_x = 0
	\end{equation}
	where the first isomorphism is \eqref{eq:stalkiso}.
	By Lemma \ref{lem:affine} the TP-stalk \eqref{eq:hix} can be computed using TP-open subsets which are open affines in the Zariski topology, and given such an open $U$, using Lemma \ref{lem:changesupp}:
	\[
	\underline H^i_Z(\phi_*\F)(U)=H^i_Z(U,\phi_*(\F))= H^i_Z(U,\F)=0.
	\]
	This equals zero since $\F$ is Cohen-Macaulay with respect to the codimension filtration in the scheme $\zar \ox$, therefore condition (1) of Proposition \ref{def:cohen} is satisfied and $i<\Codim(Z)$.
\end{proof}
\begin{corollary}
	\label{cor:exactness}
	The Cousin complex \eqref{eq:cc1} of $\O$ is a Flabby resolution of $\O$.
\end{corollary}
\begin{proof}
	The structure sheaf $\zar {\ox}$ is Cohen-Macaulay with respect to the codimension filtration in $\zar \X$ \cite[Example pg. 239]{hartshorne:residues}, therefore by Corollary \ref{cor:cohmac} its pushforward $\phi_*(\zar {\ox})=\O$ is Cohen-Macaulay with respect to the codimension filtration in $\tp \X$. This means that $\O$ satisfies condition (3) of Proposition \ref{def:cohen} and its Cousin complex is a Flabby resolution of $\O$.
\end{proof}

\subsection{Explicit description of the Cousin complex}

We want to give an explicit description of the local cohomology terms appearing in the Cousin complex of $\O$. For this task let us extend Theorem \ref{lem:affineloc} to the TP-topology.

\begin{lemma}
	\label{lem:commalg}
	Let $x$ be a point in $\tp \X$ with TP-closure $Z$, and $\I$ be the $\zar \ox$-ideal sheaf associated to $Z$. The ideal $m:=(\phi_*\I)_x$ is a well defined ideal of the ring $\tp {\oxx x}$. Then the two local cohomology functors
	\begin{equation}
		\label{eq:loccohfunctor}
		\Gamma_x(\phi_*\F)=\Gamma_m((\phi_*\F)_x)
	\end{equation}
	agree on pushforward of quasi-coherent $\zar \ox$-modules $\F$. As a consequence also their right derived functors agree on the same class:
	\begin{equation}
		\label{eq:rightderived1}
		\h *x(\phi_*\F)\cong H^*_m((\phi_*\F)_x)
	\end{equation}
\end{lemma}
\begin{proof}
	Let us first prove the containment $\Gamma_x(\phi_*\F)\subseteq \Gamma_m((\phi_*\F)_x)$. If $\alpha\in \Gamma_x(\phi_*\F)$, then by definition of $\Gamma_x$ \eqref{eq:gammax} there exists a TP-open $U$ (that by Lemma \ref{lem:affine} we can take to be affine for the Zariski-topology) and a section $s\in \F(U)$ representing the germ $\alpha$ such that $\tp \supp(s)\subseteq Z$.
	Therefore
	\begin{equation}
		\label{eq:equality}
		s\in \tp {\Gamma_Z}(U, \phi_*\F)=\zar {\Gamma_Z}(U, \F)=\Gamma_{\I(U)}(\F(U))
	\end{equation}
	where the first equality is Lemma \ref{lem:changesupp} and the second one is Theorem \ref{lem:affineloc}. By definition of $\Gamma_{\I(U)}$ \eqref{eq:gammaim} there is $n\geq 0$ such that $\I(U)^ns=0$. By Remark \ref{rem:stalk} simply invert the appropriate elements to obtain the equality $m^n\alpha=0$ in the TP-stalk at $x$.  
	
	Let us prove the other containment $\Gamma_m((\phi_*\F)_x)\subseteq \Gamma_x(\phi_*\F)$. If $\alpha\in \Gamma_{m}((\phi_*\F)_x)$, there exists $n\geq 0$ such that $m^n\alpha=0$. The ideal $m$ is finitely generated, therefore also $m^n$ is finitely generated as well. This implies that we can find a TP-open $U$ affine for the Zariski topology containing $x$ and a section $s$ representing the germ $\alpha$ such that $\I(U)^ns=0$. Using again the chain of equalities \eqref{eq:equality} the TP-support of $s$ is contained in $Z$ and since $(s,U)$ represents the germ $\alpha$ we obtain $\alpha\in \Gamma_x(\phi_*\F)$.
	
	To prove \eqref{eq:rightderived1} pick an injective resolution of $\F$ in quasi-coherent $\zar \ox$-modules \cite[\href{https://stacks.math.columbia.edu/tag/077K}{Tag 077K}]{stacks-project}. The pushforward along $\phi$ of this injective resolution is an injective resolution of $\phi_*\F$ in $\tp \ox$-modules, since by Corollary \ref{cor:exact} $\phi_*$ is exact on quasi-coherent sheaves and preserves injectives. The TP-stalk at $x$ of this last resolution is an injective resolution of $(\phi_*\F)_x$ in $\tp {\oxx x}$-modules, since taking the stalk preserves injectives. By the equality \eqref{eq:loccohfunctor} just proven we obtain the isomorphism between the respective right-derived functors \eqref{eq:rightderived1}.
\end{proof}
\begin{corollary}
	For every $ij\geq 1$ we have the isomorphism
	\begin{equation}
		\label{eq:h1dijiso}
		\h 1{\dij} (\O)\cong \quotient {\K}{\O_{\dij}}.
	\end{equation}
	With $\K$ and $\odij$ defined in \eqref{eq:k}, and \eqref{eq:dij}.
\end{corollary}
\begin{proof}
	Notice the chain of isomorphisms:
	\[
	\h 1{\dij} (\O)\cong H^1_{m_{ij}}(\O_{\dij}) \cong \frac {\O_{\dij}[\tij^{-1}]}{\O_{\dij}} = \frac {\K}{\O_{\dij}}
	\]
	where the first isomorphism is Lemma \ref{lem:commalg}, and the second one is the computation of local cohomology by means of the stable Koszul complex (see for example \cite{huneke:lectures}), since $\tij$ defined in \eqref{eq:tij} generates the principal ideal $\mij$ of those vanishing at $\dij$. 
\end{proof}
\begin{proposition}
	\label{prop:ccstalk}
	Let $F$ be a finite subgroup of $G$ and $x=\eta (\bvar F)$ be the generic (and only) point of $\bvar F$ in the Kolmogorov quotient $\tp \X$. Then the TP-stalk at $x$ of the Cousin complex \eqref{eq:cc2} is
	\begin{equation}
		\label{eq:ccstalk}
		\O_F \rightarrowtail \K \xrightarrow{d_0} \bigoplus_{i\geq 1} \quotient {\K}{\O_{D_{i,n_i}}} \xrightarrow{d_1} \h 2F(\O) \to 0.
	\end{equation}
	Moreover the sequence \eqref{eq:ccstalk} is an exact sequence of $\O_F=\O_x$-modules.
\end{proposition}
\begin{proof}
	First of all $(CC^0(\O))_x=\K$ simply computing
	\begin{equation}
		\label{eq:h0description}
		\h 0{\X}(\O)=\O_{\eta(\tp \X)}=\K.
	\end{equation}

	The next term is $(CC^1(\O))_x = \bigoplus_{i\geq 1} \quotient {\K}{\O_{D_{i,n_i}}}$, since by Lemma \ref{lem:fin} for every $i\geq 1$ the only $\dij$ containing $\bvar F$ is $D_{i,n_i}$, and the local cohomology is described in \eqref{eq:h1dijiso}.
	
	The last term is $(CC^2(\O))_x=\h 2F(\O)$ since by Lemma \ref{lem:disj} if $F'\neq F$ is another finite subgroup of $G$, then $\bvar {F'}$ and $\bvar F$ are disjoint.
	
	The sequence of $\O_F$ modules \eqref{eq:ccstalk} is exact since by Corollary \ref{cor:exactness} the Cousin complex of $\O$ is a flabby resolution of $\O$.
\end{proof}
 
We can use exactness of \eqref{eq:ccstalk} to explicitly describe also the last local cohomology term:
\begin{equation}
 	\label{eq:h2description}
 	\h 2F(\O)\cong \quotient {(\bigoplus_{i\geq 1} \quotient {\K}{\O_{D_{i,n_i}}})}{\K}.
\end{equation}

We conclude the section considering the global sections of the Cousin complex \eqref{eq:cc2}, which will provide all the geometric inputs needed later in the construction of $E\c_G$:
\begin{equation}
	\label{eq:cc3}
	\Gamma(\O) \longrightarrow \K \xrightarrow{d_0} \bigoplus_{i\geq 1} (\bigoplus_{j\geq 1} \K/\odij) \xrightarrow{d_1} \bigoplus_{F} \h 2F(\O) \rightarrow 0
\end{equation}

\section{The main construction}
\label{sec:4}
We are now ready to construct $E\c_G$. Recall that $G=\T^2$ is the 2-torus, and that we have fixed an elliptic curve $\c$ over $\C$ together with a coordinate $t_e\in \tp \O_{\c,e}\subset \O_{\c,e}$ (Choice \ref{choice:te}). We will construct $E\c_G\in d\A(G)$ from an exact sequence of injective objects in $\A(G)$:
\begin{equation}
	\label{eq:injres0}
	\begin{tikzcd}
	\inj_0 \arrow[r, "\phi_0"] & \inj_1 \arrow[r, "\phi_1"] & \inj_2 \arrow[r, "0"] & 0.
	\end{tikzcd}
\end{equation}
\begin{definition}
	\label{def:mainobj}
	Since the map $\phi_0$ factors through a map $\tilde {\phi_0}:\inj_0 \to \Fib(\phi_1)$ (where $\Fib(\phi_1)$ is the fibre of $\phi_1$ in $d\A(G)$), we can define $E\c_{G}:=\Fib(\tilde {\phi_0})$. Moreover $E\c_G$ is a formal object in $d\A(G)$.
\end{definition}
As a consequence:  
\begin{lemma}
	\label{lem:injres}
	The sequence
	\begin{equation}
		\label{eq:injres4}
		\begin{tikzcd}
		H_*(E\c_G) \arrow[r] & \inj_0 \arrow[r, "\phi_0"] & \inj_1 \arrow[r, "\phi_1"] & \inj_2 \arrow[r, "0"] & 0
		\end{tikzcd}
	\end{equation}
	is exact and it is an injective resolution of $H_*(E\c_G)$ in $\A(G)$. Where $H_*(E\c_G)$ is the homology of the differential graded object $E\c_G$.
\end{lemma}

In turn the injectives are constructed by \ref{eq:deffh}: 

\begin{equation}
	\inj_0:=f_G(V(G)) \qquad \inj_1:=\bigoplus_{i\geq 1}f_{H_i}(\bigoplus_{j\geq 1}V(H_i^j)) \qquad \inj_2:=f_1(\bigoplus_{F}V(F))
\end{equation}

for a graded injective $\hbg G$-module $V(G)$, a graded torsion injective $\hbg {H_i^j}$-module $V(H_i^j)$ for every $ij\geq 1$ and a graded torsion injective $\hbg F$-module $V(F)$ for every finite subgroup $F$ of $G$.

\begin{rem}
	Notice that the objects are indeed injective by \ref{cor:inj1} and \ref{cor:inj2}.
\end{rem}

We start in \ref{subsec:object} by detailing the construction of $E\c_G$ and proving it is a formal object to obtain Lemma \ref{lem:injres}. The bulk of the section is the explicit construction of \eqref{eq:injres0}: the objects are built in \ref{sub:inj0}, \ref{sub:inj1} and \ref{sub:inj2}, while the maps are built in \ref{sub:phi0} and \ref{sub:phi1}. We conclude proving exactness of \eqref{eq:injres0} in \ref{subsec:exact}. All the inputs needed to build \eqref{eq:injres0} come from the global sections of the Cousin complex \eqref{eq:cc3}.

\begin{nota}
	In all this section the local cohomology modules are the ones in the Cousin complex \eqref{eq:cc2}, therefore we will omit the sheaf $\O$ from the notation: $\h 2F:=\h 2F(\O)$.
\end{nota}

\subsection{Formality and main construction}
\label{subsec:object}
Starting from the exact sequence of injectives \eqref{eq:injres0} we detail the construction of $E\c_{G}$ and we prove it is formal.

Note there is a natural inclusion $\iota:\A(G)\to d\A(G)$ obtained simply regarding an object of $\A(G)$ as an object of $d\A(G)$ with zero differential. There is also another functor $H_*:d\A(G) \to \A(G)$ obtained taking the homology of the object with differential.

\begin{definition}
	An object $X\in d\A(G)$ is said to be formal when it is quasi-isomorphic to its homology $H_*(X)$.
\end{definition}
\begin{definition}
	\label{def:fibre}
	For a map $\phi:X\to Y$ in $d\A(G)$ the fibre $\Fib(\phi)\in d\A(G)$ is defined at the level $n$ by: $\Fib(\phi)_n=X_n\oplus Y_{n+1}$, with differential $d_{\Fib(\phi)}$ that on the $Y$ component is simply the differential of $Y$, and on the $X$ component is the direct sum of the differential of $X$ with the map $\phi$ itself:
	\begin{equation}
		\label{eq:fibre}
		\begin{tikzcd}[column sep= 0.1em, row sep= 2.7em]
			X_n \arrow[d, "d_X"] \arrow[rrd, "\phi_n"]& \oplus & Y_{n+1} \arrow[d,"d_Y"]\\
			X_{n-1} & \oplus & Y_{n}
		\end{tikzcd}
	\end{equation}
\end{definition}

\begin{rem}
	For a map $\phi:X \to Y$ between objects in $\A(G)$ we have two ways to consider the kernel. We have a well defined kernel $\Ker(\phi)\in \A(G)$ with no differential in the abelian category $\A(G)$, or we can consider $\phi$ as a map in $d\A(G)$ through the inclusion $\iota$ and consider the fibre of that map $\Fib(\phi)\in d\A(G)$ which is an object with differential.
\end{rem}

 When the map is surjective these two objects are equivalent:
\begin{lemma}
	\label{lem:homiso}
	If $\phi:X \to Y$ is a surjective map in $\A(G)$, then there is an homology isomorphism $\iota (\Ker(\phi))\xrightarrow{\simeq} \Fib(\phi)$.
\end{lemma} 

\begin{proof}
	Since $X$ and $Y$ have no differential, in \eqref{eq:fibre} we only have the map $\phi$ as differential. There is an obvious inclusion map $i:\Ker(\phi)\to \Fib(\phi)$ that on each level includes the kernel in the $X$-component of the fibre. It is easy to check it commutes with the differentials since the kernel has zero differential and the composition $d_{\Fib(\phi)}\circ i$ is zero as well. Since $\phi$ is surjective, the image of $d_{\Fib(\phi)}$ is the full $Y_n$ component at each level, while the kernel is $\Ker(\phi)_n\subseteq X_n$ at each level. Therefore when we take the homology of $d_{\Fib(\phi)}$, the inclusion $i$ induces an isomorphism.
\end{proof}

The following diagram helps in understanding the construction and proving formality:
\begin{equation*}
	\begin{tikzcd}
		E\c_{G} \arrow[dr] &&& \\
		\Ker(\phi_0) \arrow[r] & \inj_0 \arrow[d, "\bar {\phi_0}"']\arrow[r, "\phi_0"] \arrow[dr, "\tilde {\phi_0}"]& \inj_1 \arrow[r,"\phi_1"] & \inj_2 \\
		&\Ker(\phi_1) \arrow[r, "i"] &\Fib(\phi_1) \arrow[u] &
	\end{tikzcd}
\end{equation*}

the map $\phi_0$ factors trough a map $\tilde {\phi_0}:= i \circ \bar \phi_0$ where $\bar {\phi_0}$ is the map that $\phi_0$ induces to the kernel. Recall that $E\c_{G}$ is defined as the fibre of the map $\tilde \phi_0$

\begin{lemma}
	The object $E\c_G$ is formal. More precisely we have an homology isomorphism $\iota (\ker(\phi_0))\xrightarrow{\simeq} \Fib(\tilde {\phi_0})$
\end{lemma}
\begin{proof}
	Consider the fibre of $\bar {\phi_0}$, which fits in a commutative diagram in $d\A(G)$:
	\begin{equation*}
		\begin{tikzcd}
			\Fib(\tilde {\phi_0}) \arrow[r] & \inj_0 \arrow[r, "\tilde {\phi_0}"] & \Fib(\phi_1) \\
			\Fib(\bar {\phi_0}) \arrow[u] \arrow[r] & \inj_0 \arrow[r, "\bar {\phi_0}"]\arrow[u, "="] & \Ker(\phi_1)\arrow[u, "i", "\simeq"']
		\end{tikzcd}
	\end{equation*}
	Both rows induce a long exact sequence in homology, and applying Lemma \ref{lem:homiso} the right vertical inclusion is an homology isomorphism, since $\phi_1$ is surjective. As a consequence also the left vertical map is an homology isomorphism. Now it is enough to notice that also $\bar {\phi_0}$ is a surjective map in $\A(G)$ because the sequence is exact at $\inj_1$, therefore applying again the lemma we have an homology isomorphism $\Ker(\bar {\phi_0})\xrightarrow{\simeq} \Fib(\bar {\phi_0})$. Summing up we have a chain of homology isomorphisms:
	\[
	\Ker(\phi_0)=\Ker(\bar {\phi_0})\xrightarrow{\simeq} \Fib(\bar {\phi_0}) \xrightarrow{\simeq}\Fib(\tilde {\phi_0})=E\c_G.
	\]
	Since $\Ker(\phi_0)$ is without differential we obtain $H_*(E\c_G)=\Ker(\phi_0)$, and $E\c_G$ is formal.
\end{proof}
Assuming exactness of \eqref{eq:injres0}, we obtain as an immediate consequence Lemma \ref{lem:injres}.

\subsection{Building $\inj_0$.}
\label{sub:inj0} 
Associated to the whole group $G$ in the Cousin complex we have the codimension $0$ piece $\h 0{\X}=\K$ (defined in \eqref{eq:k}). This is a $\Q$-vector space, that we can make graded $2$-periodic. Simply consider the 2-periodic version $\overline {\K}$ that has a copy of $\K$ in each even dimension and zero in odd dimensions (in general we will always use this notation for this 2-periodic operation). This is our graded injective $\hbg G\cong \Q$-module $V(G):= \overline {\K}$:
\begin{equation}
	\label{eq:inj0}
	\inj_0:=f_G(\overline \K)
\end{equation}
\subsection{Building $\inj_1$.} 
\label{sub:inj1}
Associated to every codimension one subgroup $H_i^j$ we have the codimension one subvariety $\dij$, and the associated codimension one piece in the Cousin complex $\h 1{\dij}$. Recall from \eqref{eq:cij} the isomorphism $\hbg {H_i^j}\cong \Q[\cij]$ and the definition of the ring $\odij$  and ideal $\mij$ in \eqref{eq:odij}.

\begin{lemma}
	\label{lem:inj1}
	The module $\h 1{\dij}$ is a torsion injective $\hbg {H_i^j}\cong \Q[\cij]$-module, where the action is defined by restriction along the ring map
	\[
	\Q[c_{ij}]\to (\odij)^{\wedge}_{m_{ij}}
	\]
	that sends $\cij=e(z_i^j)$ to $\hatij$ (defined in \eqref{eq:hatij}).
\end{lemma}

To define this action we will need to switch to the completed rings. For the task we will use the following well-known fact:

\begin{lemma}
	\label{lem:completion}
	Let $I$ be a finitely generated ideal of the Noetherian ring $R$, then
	\[
	H^i_I(R)\cong H^i_{\hat I}(R^{\wedge}_I)
	\]
	where the second local cohomology is computed with respect to the completed ring $R^{\wedge}_I$.
\end{lemma}
\begin{proof}
	We have the chain of isomorphisms:
	\[
	H^i_I(R)\cong R^{\wedge}_I\otimes_RH^i_I(R)\cong H^i_{\hat I}(R^{\wedge}_I).
	\]
	Where the first isomorphism follows since the module is isomorphic to its completion. For the second isomorphism apply \cite[Proposition 2.14]{huneke:lectures}: since $R$ is Noetherian $R^{\wedge}_I$ is a flat $R$-algebra.
\end{proof}

\begin{proof}[Proof of Lemma \ref{lem:inj1}]
	Note the chain of isomorphisms:
	\begin{equation}
		\label{eq:loccohchain}
		\h 1{\dij}\cong H^1_{m_{ij}}(\odij)\cong H^1_{\hat {m}_{ij}}((\odij)^{\wedge}_{m_{ij}})\cong \frac {(\odij)^{\wedge}_{\mij}[\hatij^{-1}] }{(\odij)^{\wedge}_{\mij}}.
	\end{equation}
	The first isomorphism is Lemma \ref{lem:commalg}. The second one is Lemma \ref{lem:completion}. The third isomorphism is simply the computation of local cohomology by means of the stable Koszul complex (\cite{huneke:lectures}), since $\hatij$ generates $\hat m_{ij}$. From this chain of isomorphisms it is clear the module $\h 1{\dij}$ admits an action of $\hatij$, and also that is $\hatij$-divisible. Since $\Q[\cij]$ is a PID, divisible implies injective and $\h 1{\dij}$ is injective.
	
	The more transparent form of the module \eqref{eq:loccohchain} and the one we will use is \eqref{eq:h1dijiso}. From \eqref{eq:h1dijiso} it is immediate to see that the module is torsion, since $\hat t_{ij}$ adds a zero at $D_{ij}$ to every class, making it regular after a finite number of iterations.
\end{proof}

By the previous lemma the module $\h 1{\dij}$ is torsion injective. Consider the 2-periodic version $\overline {\h 1{\dij}}$ as before, where $\cij$ acts as an element of degree $-2$. This is the graded torsion injective $\hbg {H_i^j}$-module that we use: $V(H_i^j):= \overline {\h 1{\dij}}= \overline {\K/\odij}$.

To build $\inj_1$, define the direct sum
\begin{equation}
	\label{eq:ti}
	T_i:= \bigoplus_{j\geq 1} \overline {\K/\odij}.
\end{equation}
by Corollary \ref{cor:inj1} and Corollary \ref{cor:inj2} the object 
\begin{equation}
	\label{eq:inj1}
	\inj_1:= \bigoplus_{i\geq 1}f_{H_i}(T_i)
\end{equation}
is injective and well defined in $\A(G)$.

\subsection{Building $\inj_2$} 
\label{sub:inj2}
Associated to every finite subgroup $F$ of $G$ we have the codimension two subvariety $\bvar F$, and the codimension two piece in the Cousin complex $\h 2F$. Recall from \eqref{eq:xaxb} the isomorphism $\hbg F\cong \Q[\xa,\xb]$, and the definition of the ring $\O_F$ and ideal $m_F$ in \eqref{eq:of}.
\begin{lemma}
	\label{lem:action2}
	The module $\h 2F$ is a torsion injective $\hbg F\cong \Q[\xa,\xb]$-module, where the action is defined by restriction along the ring map
	\[
	\Q[\xa,\xb]\to (\O_{F})^{\wedge}_{m_F}
	\]
	that sends $\xa=e(\za)$ to $\hatone$, and $\xb=e(\zb)$ to $\hattwo$ (defined in \eqref{eq:hatij}).
\end{lemma}
\begin{proof}
	We have the chain of isomorphisms:
\begin{equation}
	\label{eq:locmod}
	\h 2F\cong H^2_{m_F}(\O_F)\cong H^2_{\hat {m}_F}((\O_{F})^{\wedge}_{m_F}).
\end{equation}
The first isomorphism is Lemma \ref{lem:commalg}, while the second one is Lemma \ref{lem:completion}. From this chain of isomorphisms it is immediate that $\h 2F$ is an $(\O_{F})^{\wedge}_{m_F}$-module, and therefore we can define the action of $\xa$ and $\xb$ as the action of $\hatone$ and $\hattwo$ since $m_F=\gen {\ta, \tb}$. Moreover the module is torsion because ${\hat m}_F=\gen {\hatone,\hattwo}$.

To prove that $\h 2F$ is an injective module we will split it as a direct sum of local cohomology modules for the Zariski topology and then use regularity. Consider a TP-open $\Spec(R)$ (that we can take to be affine in the Zariski topology)  containing $\bvar F=\{Q_1, \dots Q_k\}$ (which is a finite collection of closed points in $\zar {\X}$). Modulo restricting the open we can consider $m_i$ to be the maximal ideal in $R$ associated to the closed point $Q_i$. By Remark \ref{rem:stalk}: $\O_{F}=S^{-1}R$ for a multiplicatively closed subset $S$, therefore in $\O_{F}$ the ideal $S^{-1}m_i$ remains maximal, and we can write $m_F=S^{-1}m_1\cdot S^{-1}m_2\dots S^{-1}m_k$.
We can now apply Mayer-Vietoris for local cohomology \cite[Theorem 2.3]{huneke:lectures} since $\{Q_1, \dots Q_k\}$ are disjoints, obtaining the desired splitting:
\[
H^2_{m_F}(\O_{F})\cong H^2_{S^{-1}m_1}(\O_{F})\oplus \dots \oplus H^2_{S^{-1}m_k}(\O_{F}).
\]

We prove now that each factor $H^2_{S^{-1}m_i}(\O_{F})$ is injective, so that the direct sum is injective as well. First of all notice that by Remark \ref{rem:stalk} if we localize at $S^{-1}m_i$ we obtain the local ring for the Zariski structure sheaf at the point $Q_i$:
\[
(\O_{F})_{S^{-1}m_i}\cong R_{m_i} = \zar {\oxx {Q_i}}.
\]
Since the ideal $S^{-1}m_i$ is maximal we obtain:
\begin{equation}
	\label{eq:locpiece}
	H^2_{S^{-1}m_i}(\O_{F}) \cong H^2_{m_i}(R_{m_i}) \cong H^2_ {\hat {m_i}}((R_{m_i})^{\wedge}_{m_i})
\end{equation}
where first we localize at $S^{-1}m_i$ obtaining local cohomology of a local ring with respect to its maximal ideal \cite[Proposition 8.1]{hochster:local}, and then we complete with respect to that maximal ideal. The final module of \eqref{eq:locpiece} is injective since $\zar \X$ is regular. To prove this notice that $(R_{m_i})^{\wedge}_{m_i}$ is a completed Noetherian local ring whose residue field is $\C$, that is therefore isomorphic to a ring of power series $\C[[\xa,\xb]]$ (here we can use $\xa$ and $\xb$ since through the action just defined they are sent to two generators of $\hat m_i$). The module \ref{eq:locpiece} is injective over $\C[[\xa,\xb]]$: the base ring being regular local of dimension $2$ is in particular Gorenstein local of dimension $2$ and therefore its top local cohomology $H^2_{(\xa,\xb)}(\C[[\xa,\xb]])$ is an injective hull of its residue field \cite[Proposition 11.8]{hochster:local}. The module \eqref{eq:locpiece} remains injective over $\Q[\xa,\xb]$ since the scalar extension to $\C[[\xa,\xb]]$ is faithfully flat.
\end{proof}
By the previous lemma the module $\h 2F$ is torsion injective. Consider the 2-periodic version $\overline {\h 2F}$ as before, where $\xa$ and $\xb$ act as elements of degree $-2$. This is the graded torsion injective $\hbg F$-module that we use: $V(F):=\overline {\h 2F}$.

To build $\inj_2$, simply define the direct sum over all the finite subgroups
\begin{equation}
	\label{eq:n}
	N:= \bigoplus_{F}\overline {\h 2F}
\end{equation}
by Corollary \ref{cor:inj2} the object
\begin{equation}
	\label{eq:inj2}
	\inj_2:=f_1(N)
\end{equation} 
is injective and well defined in $\A(G)$.

\subsection{Building $\phi_0$.} 
\label{sub:phi0}
Now we have constructed the relevant injective objects, we turn to constructing maps 
\begin{equation}
	\label{eq:injres1}
	f_G(\overline \K) \xrightarrow{\phi_0} \bigoplus_{i\geq 1}f_{H_i}(T_i) \xrightarrow{\phi_1} f_1(N) \rightarrow 0
\end{equation}
between them. Recall the definition of $\K$ \eqref{eq:k}, $T_i$ \eqref{eq:ti}, and $N$ \eqref{eq:n}.

Each component $\phi_0^i$ is determined by an $\oefh i$-map
\[
\phi_0^i:\eghinv i\oefh i \otimes \ov {\K}\to T_i
\]
tensored with $\ehinv i\of$ (Where the rings are defined in \eqref{eq:ofhi}). Moreover we want this map to extend the one in the Cousin complex \eqref{eq:cc3}: 
\begin{lemma}
	\label{lem:phi0}
	For every $i\geq 1$ there exists an $\oefh i$-map $\phi_0^i$, making the following diagram commute:
	\begin{equation}
		\begin{tikzcd}
			\label{eq:comm0}
			\ov {\K} \arrow[r,"d_0^i"] \arrow[d, "1\otimes \_"]& T_i\\
			\eghinv i \oefh i \otimes \ov {\K} \arrow[ur, dotted, "\phi_0^i"']&
		\end{tikzcd}
	\end{equation}
where $d_0^i$ is the $i$-th component of the map in the Cousin complex \eqref{eq:cc3}.
\end{lemma}
\begin{proof}
	The first step is to extend $d_0^i$ to an $\oefh i$-map: $\oefh i \otimes \ov {\K} \to T_i$. This is completely determined by the action of $\oefh i$ on the target and the fact that $d_0^i$ is a $\Q$-map.
	
	We need to extend it further to the localization $\eghinv i$, and we can define it for every $j$-th component $\hbg {H_i^j}\cong \Q[\cij]$  \eqref{eq:cij}. Notice that inverting $\egh i$ in $\Q[\cij]$ means inverting the Euler class $\cij=e(z_i^j)$ (Example \ref{ex:eulerclasses}). Therefore extend the map for negative powers of $\cij$ in the following way:
	\begin{equation}
		\label{eq:phi0}
		\begin{split}
			(\phi_0^i)_j:\Q[\cij^{\pm 1}]\otimes \overline\K &\to \overline{\K/\odij}\\
			\cij^{-k}\otimes f &\mapsto [\hat{t}_{ij}^{-k}\cdot f]
		\end{split}
	\end{equation}
	where $\hat{t}_{ij}^{-1}$ is the power series inverse of $\hat t_{ij}$ (see \eqref{eq:hatijpowerseries}, \eqref{eq:hatij}, and \eqref{eq:tij}) with complex coefficients:
	\begin{equation}
		\label{eq:powinv}
		\hat t_{ij}^{-1}=t_{ij}^{-1}+a_0+a_1t_{ij}+a_2t_{ij}^2+\dots
	\end{equation}
	and the ring $\odij$ and ideal $\mij$ are defined in \eqref{eq:odij}.
	The map in \eqref{eq:phi0} is well defined since for every element in the target there is a power of the ideal $\mij$ that annihilates that element and $\mij = \gen {\tij}$. Moreover the map in the Cousin complex
	\[
	d_0^i:\K\to \quotient{\K}{\odij}
	\]
	is an $\odij$-module map, so that after a certain power of $t_{ij}$ the terms in \eqref{eq:powinv} do not contribute and the sum is finite.
\end{proof}

\subsection{Building $\phi_1$.}
\label{sub:phi1}
The map $\phi_1$ is also determined by its components
\[
\phi_1^i:\ehinv i\of \tens {\oefh i} T_i \to N.
\]
Recall the definition of $T_i$ \eqref{eq:ti}, $N$ \eqref{eq:n}, and the notation \ref{notation:tensi} $\tens i= \tens {\oefh i}$. As before we want this map to extend the one in the Cousin complex \eqref{eq:cc3}:
\begin{lemma}
	\label{lem:phi1}
	For every $i\geq 1$ there exists an $\of$-map $\phi_1^i$, making the following diagram commute:
	\begin{equation}
		\label{eq:comm1}
		\begin{tikzcd}
			T_i \arrow[r, "d_1^i"] \arrow[d, "1\otimes \_"] & N  \\
			\ehinv i \of \tens {i} T_i \arrow[ur, dotted, "\phi_1^i"']&
		\end{tikzcd}
	\end{equation}
	where $d_1^i$ is the $i$-th component of the map in the Cousin complex \eqref{eq:cc3}.
\end{lemma}
\begin{rem}
	First of all $d_1^i$ should be an $\oefh i$-map. This is more delicate than it looks, we have defined an $\oefh i$-action on $T_i$ in Lemma \ref{lem:inj1} and an $\of$-action on $N$ in Lemma \ref{lem:action2} using specific coordinates for every subgroup $F$. A priori it is not guaranteed that $d_1^i$ coming from the Cousin complex commutes with the action of $\oefh i$ on the source and on the target. Here we are considering $\oefh i$ acting on $N$ by restriction along the ring map $\oefh i \to \of$ of Remark \ref{rem:inflationmap}. Nonetheless with our choice of coordinates this is indeed the case. In the argument it is crucial that $\hat t_e$ is an isomorphism of the formal group law of $\c$ with the additive formal group law.
\end{rem}

\begin{lemma}
	\label{lem:ccmap}
	The $i$-th component of $d_1$ in \eqref{eq:cc3}:
	\begin{equation}
		\label{eq:d1i}
		d_1^i:T_i \to N
	\end{equation}
	is an $\oefh i$-map with respect to the actions defined in Lemma \ref{lem:inj1} and Lemma \ref{lem:action2}.
\end{lemma}
\begin{proof}
	To simplify the notation in this proof we omit the bar to indicate the 2-periodic version of the various modules. First, the $j$-th component in the source is $\h 1{D_{ij}}$ and it is sent into the $j$-th component in the target: the direct sum on all $F$ finite such that $F$ and $H_i$ generate $H_i^j$:
	\[
	\h 1{\dij} \to \bigoplus_{\gen {F,H_i}=H_i^j}\h 2F.
	\]
	
	This is because the induced map \eqref{eq:d1i} $\h 1{D_{ij}}\to \h 2F$ is non zero if and only if $\bvar F\subseteq D_{ij}$ if and only if $\gen {F,H_i}=H_i^j$ by Lemma \ref{lem:fin}. Fixed $F$, by the same lemma the map is non-zero only at the index $j=n_i$.
	
	Therefore it is enough to show that for any finite subgroup $F$ the induced map \eqref{eq:d1i}
	\begin{equation}
		\label{eq:cousmap}
		d_1^i:\h 1{D_{i,n_i}}\to \h 2F
	\end{equation}
	
	is an $\hbg {H_i^{n_i}}\cong \Q[c_{i,n_i}]$-map, namely it commutes with the action of the Euler class $c_{i,n_i}=e(z_i^{n_i})$ on the source and on the target. 
	
	By Lemma \ref{lem:inj1}: $c_{i,n_i}$ acts as $\hat t_{i,n_i}$ on $\h 1{D_{i,n_i}}$. On $\h 2F$ the action is defined in Lemma \ref{lem:action2}: $\h 2F$ is an $\hbg F\cong \Q[x_A,x_B]$-module, and $\hbg {H_i^{n_i}}$ acts by restriction along the $F$th component of the inflation map \eqref{eq:inflationhi}. Therefore $c_{i,n_i}$ acts as
	\begin{equation}
	    \label{eq:xi2}
	    x_i=r \cdot \xa +s \cdot \xb
	\end{equation}
	defined in \eqref{eq:xi}, where $r$ and $s$ are the two integers such that
	\begin{equation}
		\label{eq:chareq}
		z_i^{n_i}=(\za)^r\cdot(\zb)^s
	\end{equation}
    in the character group of $G/F$.
	The equality \eqref{eq:xi2} can be obtained by taking Euler classes on both members of \eqref{eq:chareq} and noticing that the group operation in the character group translates into sum in $\Q[x_A,x_B]$. Therefore by Lemma \ref{lem:action2} the Euler class $c_{i,n_i}$ acts on $\h 2F$ as $r \cdot \hatone +s \cdot \hattwo$.
	
	We are only left to show that these two actions commute with \eqref{eq:cousmap}. This translates in proving that the equality
	\begin{equation}
		\label{eq:fundrel}
		\hat t_{i,n_i}= r\cdot \hatone +s\cdot \hattwo
	\end{equation}
 	holds in $(\O_{F})^{\wedge}_{m_F}$, since $\hat t_{i,n_i}$ commutes with the $\O_{F}$-module map \eqref{eq:cousmap} (recall the definition of the ring $\O_F$ and the ideal $m_F$ in \eqref{eq:of}).
	
	To prove \eqref{eq:fundrel} apply the functor $\XF$ to the characters equality  \eqref{eq:chareq}, to obtain the same linear relation with the projections (Definition \ref{def:zij}):
	\[
	\pi_i^{n_i}=r\cdot \pia+s\cdot \pib.
	\]
	Note that in this last equality the group operation is the one of $\c$. Recalling the definition of the completed coordinates \eqref{eq:hatij}, combined with Lemma \ref{lem:lincomb} we obtain precisely:
	\[
	\hat t_{i,n_i}=(\pi_i^{n_i})^*(\hat t_e)= (\pia,\pib)^*\circ (r,s)^*(\hat t_e)=r\cdot \hatone +s\cdot \hattwo
	\]
\end{proof}

\begin{rem}
	This proof clarifies the need of using completed coordinates $\hat t_{i,n_i}$ instead of the uncompleted ones $t_{i,n_i}$. For the uncompleted functions the relation \eqref{eq:fundrel} is not true (compute zeroes and poles on both sides). As a consequence the algebra of the Euler classes simply does not match the algebra of the uncompleted functions.
\end{rem}

\begin{proof}[Proof of Lemma \ref{lem:phi1}]
	Proceed as we did in Lemma \ref{lem:phi0} and first extend $d_1^i$ to an $\of$-map 
	\[
	\of \tens {\oefh i} T_i  \to N.
	\]
	This is completely determined by the action of $\of$ on $N$ and Lemma \ref{lem:ccmap}. 
	
	We need to extend it further to the localization $\ehinv i$, and we can define it for every $F$th component of $\of$: $\hbg F \cong \Q[\xa,\xb]$. Recall from Example \ref{ex:eulerclasses} that inverting $\ehinv i$ in $\Q[\xa,\xb]$ means inverting all the Euler classes $x_j$ with $j\geq 1$ and $j\neq i$. Therefore define the map for negative powers of the Euler classes in the following way:
	\begin{equation}
		\label{eq:phi1}
		\begin{split}
			(\phi_1^i)_F: \ehinv i\Q[\xa,\xb]\tens {\Q[x_i]} \overline {\quotient{\K}{\O_{D_{i,n_i}}}} &\to \overline {\h 2F}\\
			\frac 1{x_1^{k_1}x_2^{k_2}\dots x_r^{k_r}}\tens {\Q[x_i]}[f] &\mapsto d_1^i([\hat t_{1,n_1}^{-k_1} \cdot \hat t_{2,n_2}^{-k_2} \dots \hat t_{r,n_r}^{-k_r} \cdot f])
		\end{split}
	\end{equation}
	and extend it to be a $\Q[\xa,\xb]$-module map (see \eqref{eq:powinv} for the power series inverses).
	\begin{rem}
		\label{rem:capped}
		The main issue with this definition is that the element
		\begin{equation}
			\label{eq:inftysum}
			\hat t_{1,n_1}^{-k_1} \cdot \hat t_{2,n_2}^{-k_2} \dots \hat t_{r,n_r}^{-k_r}
		\end{equation}
		is a priori an infinite sum of monomials $t_{1,n_1}^{j_1}\cdot \tin 2^{j_2}\dots \tin r^{j_r}$ and not a well defined element in $\K$. To address this issue note that every element in $\h 2F$ is annihilated by a power of the ideal $m_F$, that every coordinate $t_{j,n_j}$ is in $m_F$, and that $d_1^i$ is an $\O_{F}$-module map. Therefore there exists a positive integer $d$ such that in the infinite sum \eqref{eq:inftysum}, elements of total degree higher than $d$ do not contribute in any way in \eqref{eq:phi1}. In this way we can cap the infinite sum \eqref{eq:inftysum} accordingly so that the sum is finite and therefore is a well defined element of $\K$.
	\end{rem}
	It is immediate now to check \eqref{eq:phi1} is compatible with the already defined action of the positive powers of $x_j$ for $j\neq i$. Namely $t_{j,n_j}$ commutes with $d_1^i$, and they simplify with the capped $\hat t_{j,n_j}^{-1}$ anyway since we are left with terms of higher enough degree to not contribute to the result. 
	
	It is immediate also to check that the definition \eqref{eq:phi1} does not depend on the representative $f$ picked for the class, since if $[f+g]$ is another representative with $g$ regular on $D_{i,n_i}$, then $g\cdot \hat t_{j,n_j}^{-1}$ is regular on $D_{i,n_i}$ for any $j\neq i$ and does not contribute in any way.
\end{proof}

\subsection{Exactness}
\label{subsec:exact}
We turn now to prove that the sequence \eqref{eq:injres0} is exact. From Remark \ref{rem:bottomlevel} we only need to check exactness at the bottom level (the level at the trivial subgroup). Therefore it is enough to prove that the bottom level of \eqref{eq:injres0}:
\begin{equation}
	\label{eq:injresbott}
	\eginv \of \otimes \overline \K \xrightarrow{\phi_0} \bigoplus_{i\geq 1}\ehinv i\of \tens i T_i \xrightarrow{\phi_1} N\to 0
\end{equation}
is an exact sequence of $\of$-modules. We can do it for each $F$th component of $\of$ at a time. Namely prove that for every finite subgroup $F$ of $G$, the $F$th component of \eqref{eq:injresbott}:
\begin{equation}
	\label{eq:injrescomp}
	\eginv \Q[\xa, \xb] \otimes \overline \K \xrightarrow{\phi_0} \bigoplus_{i\geq 1}\ehinv i \Q[\xa, \xb] \tens i \overline {\quotient {\K}{\O_{D_{i,n_i}}}} \xrightarrow{\phi_1} \overline {\h 2F} \to 0.
\end{equation}
is an exact sequence of $\hbg F\cong \Q[\xa, \xb]$-modules. 

The strategy is to use exactness of \eqref{eq:ccstalk}. For this task we will need the ``Moving'' Lemma \ref{lem:reduction}. In the interest of the exposition we start in proving exactness of \eqref{eq:injrescomp} and will devote the rest of the section in proving Lemma \ref{lem:reduction}.

\begin{lemma}
	\label{lem:exactness}
	The sequence \eqref{eq:injrescomp} is an exact sequence of $\Q[\xa,\xb]$-modules.
\end{lemma} 
\begin{proof}
	We use exactness of the sequence \eqref{eq:ccstalk} which is the first row of the following commutative diagram:
	\begin{equation}
		\label{eq:exactdiagram}
		\begin{tikzcd}
		 \overline{\K} \arrow[r, "d_0"] \arrow[d, "1\otimes\_"] & \bigoplus_{i\geq 1} \overline{\quotient {\K}{\O_{D_{i,n_i}}}} \arrow[r, "d_1"] \arrow[d, "\iota"]&  \overline{\h 2F} \arrow[r] \arrow[d, "\Id"] & 0 \\
		 \eginv \Q[\xa, \xb] \otimes \overline \K \arrow[r, "\phi_0"] &
		 \bigoplus_{i\geq 1}\ehinv i \Q[\xa, \xb] \tens i \overline {\quotient {\K}{\O_{D_{i,n_i}}}} \arrow[r,"\phi_1"] &
		 \overline {\h 2F} \arrow[r] &
		 0
		\end{tikzcd}
	\end{equation}
	The vertical map $\iota$ is the natural inclusion $\alpha \mapsto 1\tens i \alpha$ on every $i$-th component. Note that the commutativity of \eqref{eq:exactdiagram} is due to the commutativity of \eqref{eq:comm0} and \eqref{eq:comm1}.

	First of all the composition $\phi_1\circ \phi_0$ is zero. This is because on pure tensor elements in the source:
	\[
	\phi_1\circ \phi_0(\frac 1{x_1^{k_1}x_2^{k_2}\dots x_r^{k_r}} \otimes f)=d_1\circ d_0(\hat t_{1,n_1}^{-k_1}\cdot \hat t_{2,n_2}^{-k_2} \dots \hat t_{r,n_r}^{-k_r}\cdot f)=0.
	\]
	The first equality comes from the two definitions \eqref{eq:phi0}, and \eqref{eq:phi1}. This all equals zero since the composition $d_1\circ d_0$ is zero in the first row.
	
	The map $\phi_1$ is surjective being an extension of $d_1$, which is surjective.
	
	We are left to prove that $\Ker(\phi_1)\subseteq \im(\phi_0)$. Pick $X\in \Ker(\phi_1)$. By Lemma \ref{lem:reduction}, modulo $\im(\phi_0)$ we can suppose $X\in \im(\iota)$. We can now conclude with diagram chasing: it exists $Y$ such that $\iota(Y)=X$, $d_1(Y)=0$ and since the first row of \eqref{eq:exactdiagram} is exact it also exists $F\in \overline {\K}$ such that $d_0(F)=Y$. Therefore $\phi_0(1\otimes F)=X$.
\end{proof}

We devote the rest of this section in proving the ``Moving'' lemma:
\begin{lemma}
	\label{lem:reduction}
	In the diagram \eqref{eq:exactdiagram}, every element
	\[
	X \in \bigoplus_{i\geq 1}\ehinv i \Q[\xa, \xb] \tens i \overline {\quotient {\K}{\O_{D_{i,n_i}}}}
	\]
	is equivalent, modulo the image of $\phi_0$ to an element in the image of $\iota$.
\end{lemma}

To prove this technical lemma we will need to work with the irreducible components $\dip$ of the $\dij$ in the Zariski topology, and build different uniformizers for them. We will also require somme lemmas about the specific geometry of $\X$. 

We start recalling Abel's Theorem \cite[Corollary 3.5]{silverman}:

\begin{lemma}[Abel's Theorem for elliptic curves]
	\label{lem:abel}
	Let $\c$ be an elliptic curve with identity element $e$, and let $D=\sum n_P(P)$ be a divisor of $\c$. Then there exists a meromorphic function $f$ such that $\Div(f)=D$ if and only if
	\[
	\sum n_P=0 \quad \text{and} \quad \sum n_PP=e.
	\]
	Where the first sum is in $\Z$ and the second one is addition in $\c$.
\end{lemma}

\begin{lemma}
	\label{lem:coordpoints}
	Given $P\in \c\gen n$  and a direction $i\geq 1$ there exists $h_{i,P}\in \K$ such that: 
	\begin{enumerate}
		\item for $n\neq 1$ the function $h_{i,P}$ has zeroes only at $\dip$ and $D_i$ both at first order.
		\item for $n=1$ the function $h_{i,e}$ has zeroes only at $D_i$ and $D_{i,Q}$ for $Q$ a point of exact order $2$, both at first order.
	\end{enumerate}
\end{lemma}
\begin{proof}
	If $P=e$ and $\{P_1,P_2,P_3\}$ are the three points of $\c\gen 2$ then the divisor 
	\[
	D_e:=(e)+(P_1)-(P_2)-(P_3)
	\]
	satisfies Lemma \ref{lem:abel} so there is a meromorphic function $h_e$ with divisor $D_e$. If we now pullback $h_e$ along the projection $\pi_i$ the function
	\[
	h_{i,e}:= \pi_i^*(h_e)\in \K
	\]
	satisfies (2).
	
	In the same way if $P\neq e$ pick a point $P'$ such that $2P'=P$, then the divisor
	\[
	D_P:=(e)+(P)-2(P')
	\]
	satisfies Lemma \ref{lem:abel} so there's a meromorphic function $h_P$ with divisor $D_P$. As before the pullback along the projection $\pi_i$ the function
	\[
	h_{i,P}:= \pi_i^*(h_P) \in \K
	\]
	satisfies (1).
\end{proof}
\begin{lemma}
	\label{lemma:julian}
	Consider two different arbitrary directions $a,b\geq 1$ such that $D_a$ and $D_b$ intersect only in the origin $O=\{e\}\times \{e\}$ of $\X$. Given a class $[f]\in \K/\O_{D_{a,n}}$ for $n\geq 1$, there exists a representative $f\in \K$ of that class such that $f$ has poles only along the $a$ and $b$ directions (i.e. it's regular on all $D_{ij}$ such that $i\neq a,b$).
\end{lemma}
\begin{proof}
	Since $D_a$ and $D_b$ intersect only in $O$, we can change coordinates (autoisogeny) on $\X$ in such a way that $D_a=D_1=\{e\}\times \c$ and $D_b=D_2=\c\times \{e\}$. 
	
	For any point $P$ of finite order of $\c$ and direction $i\geq 1$, denote $\O_{D_{i,P}}$ the subring of $\K$ of those functions that are regular at $\dip$, and $m_{i,P}$ the ideal of those functions that vanishes at $\dip$. The pullback along the projection $\pi_2: \X \to \c$ induces an isomorphism
	\begin{equation}
		\label{eq:coeff}
		\pi_2^*:\K(\tp \c) \xrightarrow {\cong} \quotient {\O_{D_{1,P}}}{m_{1,P}}
	\end{equation}
	where $\K(\tp \c)$ is the ring of meromorphic functions of $\tp {\c}$ (Definition \ref{def:tptopc}), namely those meromorphic functions of $\c$ that have poles only at points of finite order of $\c$. First notice that with the TP-topology on $\c$ we pullback functions in $\K$ that are regular on $D_{1,P}$. To prove injectivity of \eqref{eq:coeff} consider the inclusion $\i {1,P}:D_{1,P}\cong \c \hookrightarrow \X$ and the pullback
	\[
	\quotient {\O_{D_{1,P}}}{m_{1,P}} \xrightarrow {\i {1,P}^*} \K(\tp \c).
	\]
	Composing the two maps $\i {1,P}^*\circ \pi_2^*=(\pi_2\circ \i {1,P})^*=\Id_{\c}^*$, so $\pi_2^*$ is injective. To prove surjectivity of \eqref{eq:coeff} pick $f\in \O_{D_{1,P}}$, then
	\[
	f-\pi_2^*\circ \i {1,P}^*(f) \in m_{1,P}
	\]
	because $\i {1,P}^*(f-\pi_2^*\circ \i {1,P}^*(f))=\i {1,P}^*(f)-\i {1,P}^*(f)=0$.
	
	Consider now a class $[f]\in \K/\O_{D_{1,n}}$ and an arbitrary representative $f\in \K$ of that class. If the class is the trivial one, pick any pullback $\pi_2^*(g)$ as representative. Otherwise $f$ is not regular on $D_{1,n}$, which means it has a pole on some of its components $D_{1,P}$ with $P\in \c\gen n$. Enumerate those components $\{D_{1,P_1}, D_{1,P_2}, \dots ,D_{1,P_r}\}$, and suppose $f$ has a pole of order $k_j\geq 1$ on $D_{1,P_j}$. Expand $f$ into its principal parts at every pole $D_{1,P_j}$ using coefficients in $\quotient {\O_{D_{1,P_j}}}{m_{1,P_j}}\cong \K(\tp \c)$ and the uniformizer $h_{1,P_j}$ of $D_{1,P_j}$ (By Lemma \ref{lem:coordpoints}: $m_{1,P_j}=\gen {h_{1,P_j}}$). Recursively:
	\begin{equation}
		\label{eq:step1}
		f\cdot h_{1,P_1}^{k_1}\in \O_{D_{1,P_1}}\setminus m_{1,P_1}
	\end{equation}
	therefore using the isomorphism \eqref{eq:coeff} there exists $g\in \K(\tp \c)$ whose pullback represents the same class modulo the ideal $m_{1,P_1}$:
	\[
	f\cdot h_{1,P_1}^{k_1}\equiv \pi_2^*g\quad (m_{1,P_1})
	\]
	so that dividing by $h_{1,P_1}$:
	\[
	f\cdot h_{1,P_1}^{k_1-1}- \frac{\pi_2^*g}{h_{1,P_1}} \in  \O_{D_{1,P_1}}\setminus m_{1,P_1}
	\]
	we can reapply \eqref{eq:step1} until we obtain an expansion at the pole $D_{1,P_1}$ with coefficients that are pullbacks of functions $g_{i,1}\in \K(\tp \c)$:
	\begin{equation}
		\label{eq:recursiv}
		f- \sum_{i=-k_1}^{-1}\frac{\pi_2^*(g_{i,1})}{h_{1,P_1}^i} \in  \O_{D_{1,P_1}}
	\end{equation}
	Now move on to the next pole $D_{1,P_2}$ and do the same using \eqref{eq:recursiv} instead of $f$. Notice that by Lemma \ref{lem:coordpoints} $h_{1,P_j}$ does not vanish at $D_{1,P_i}$ for $P_j\neq P_i$ both in $\c\gen n$, in this way we can deal with the pole at each $D_{1,P_j}$ separately without changing the order of pole on the other components. 
	
	Continue in this way expanding all the poles $\{D_{1,P_1},D_{2,P_2}, \dots , D_{r,P_r}\}$ until we have something that is regular on the whole $D_{1,n}$:
	\begin{equation}
		\label{eq:refinal}
		f-\sum_{i=-k_1}^{-1}\frac{\pi_2^*(g_{i,1})}{h_{1,P_1}^i}-\dots -\sum_{i=-k_r}^{-1}\frac{\pi_2^*(g_{i,r})}{h_{1,P_r}^i} \in  \O_{D_{1,n}}.
	\end{equation}
	This gives us an explicit function in the same class of $f$ modulo $\O_{D_{1,n}}$, but with poles only along the directions $1$ and $2$.
\end{proof}

\begin{proof}[Proof of Lemma \ref{lem:reduction}]
	Doing it for one $i$-component at a time we can suppose $X$ has only one component different from zero. Therefore $X$ has all components equal to zero except the $s$th component:
	\begin{equation}
		\label{eq:x}
		X_s = \frac {1}{x_{c_1}^{k_1}x_{c_2}^{k_2}\dots x_{c_r}^{k_r}}\tens s [f] \in \ehinv s \Q[\xa, \xb] \tens s \overline {\quotient {\K}{\O_{D_{s,n_s}}}}
	\end{equation}
	where at the denominator we have inverted $r\geq 0$ Euler classes $x_{c_1}, \dots x_{c_r}\in \eh s$ (Example \ref{ex:eulerclasses}) with $c_j\geq 1$ for $1\leq j\leq r$. Let us prove the lemma by induction on $r$.
	If there are no variables at the denominator ($r=0$), then $X$ is in the image of $\iota$. Therefore we only need to prove we can reduce $r$ by one.
	
	Let us first do the case when one of the directions $\{D_{c_1},\dots ,D_{c_r}\}$ intersects $D_s$ only in the origin $O=\{e\}\times \{e\}$, and without lost of generality suppose it is the direction $D_{c_1}$. Applying Lemma \ref{lemma:julian} we can find a representative $f\in \K$ of the class $[f]\in \quotient {\K}{\O_{D_{s,n_s}}}$ such that $f$ has poles only along the directions $s$ and $c_1$. Consider the element
	\[
	X':= \frac {1}{x_{c_1}^{k_1}x_{c_2}^{k_2}\dots x_{c_r}^{k_r}}\otimes f \in \eginv \Q[\xa,\xb] \otimes \overline{\K}.
	\]
	Then the element $X-\phi_0(X')$: 
	\begin{itemize}
		\item It has a zero in the component $s$.
		\item In the component $c_j$, with $1 \leq j \leq r$ it has less than $r$ Euler classes inverted at the denominator, since $x_{c_j}$ itself is not inverted any more.
		\item In all the other components it has a zero, since $f$ has poles only along the directions $s$ and $c_j$.
	\end{itemize}
	
	In case none of the directions $\{D_{c_1},\dots ,D_{c_r}\}$ intersects $D_s$ trivially, simply pick an extra direction $D_{d}$ that intersects $D_s$ and $D_{c_1}$ only in $O$. Now apply the same argument twice. The first time use $D_d$ instead of $D_{c_1}$ and notice that we have decreased the Euler classes inverted in all components except in the component $d$. Then apply the same argument with the component $d$ instead of $s$, noticing that now $D_d$ and $D_{c_1}$ intersect in only one point.
	
	We can apply the same argument in case of a numerator in the component $s$ of $X$ \eqref{eq:x} different from $1$. Note that $\Q[\xa, \xb]=\Q[x_s,x_{c_1}]$, and distributing sums and products we are left with the case
	\[
	X_s = x_{c_1}^k \tens s [f]
	\]
	with $D_{c_1}$ intersecting $D_s$ only in $O$. Apply again Lemma \ref{lemma:julian} with the directions $s$ and $c_1$ and proceed exactly as before.
\end{proof}

\section{Values on spheres of complex representations}
\label{sec:5}
The goal of this section is to determine the values of $E\c_{G}$ (constructed in Section \ref{sec:4} for the 2-torus $G=\T^2$) on spheres of complex representations. Here resides the connection with the geometry of the curve: 
\begin{theorem}
	\label{thm:computationspheres}
	If $V$ is a finite dimensional complex representation of $G$ with $V^G=0$, then:
	\begin{equation}
		\label{eq:valthe}
		E\c_{G}^n(S^V)\cong 
		\begin{cases*}
			H^0(\X, \odv)\oplus H^2(\X, \odv) & $n$ even \\
			H^1(\X, \odv) & $n$ odd.
		\end{cases*}.
	\end{equation}
\end{theorem}
Where $\X=\c\times \c$, $\c$ is our fixed elliptic curve, and $D_V$ is a divisor of $\X$ defined as follows:
\begin{definition}
	\label{def:dv}
	If $V$ is a finite dimensional complex representation of $G$ with $V^G=0$ and dimension function:
	\begin{equation}
		\label{eq:vij}
		v_{ij}:=\Dim_{\mathbb{C}}(V^{H_i^j}) \qquad ij\geq 1
	\end{equation}
	where $H_i^j$ is the subgroup of $G$ with $j$ connected components and identity component $H_i$, then
	\begin{equation}
		\label{eq:dv}
		D_V=\sum_{i,j\geq 1}v_{ij}D_{ij}
	\end{equation}
\end{definition}
\begin{rem}
	Notice that if we write $V$ using characters $z_i^n$ (see Definition \ref{def:zij}):
	\[
	V=\sum_{i\geq 1}\sum_{n\neq 0}a_{i,n}z_i^n
	\]
	then the associated divisor is obtained simply applying the functor $\XF$ (see \eqref{eq:char} and \eqref{eq:pii}): 
	\[
	D_V:=\sum_{i\geq 1}\sum_{n\neq 0}a_{in}\var {H_i^{\modulo n}}.
	\]
\end{rem}
\begin{definition}
	Given a divisor $D_V$ of the kind \eqref{eq:dv}, define the sheaf $\odv$ on $\tp \X$ twisted by the line bundle $D_V$ to be the subsheaf of the structure sheaf $\O=\tp \ox$ with values on TP-opens $U$:
	\[
	\odv (U)=\{f\in \K \mid \Div(f)-D_V\geq 0 \;\text{on}\; U\}.
	\]
	(Recall the definition of $\K$ \eqref{eq:k}).
\end{definition}

The strategy to prove \eqref{eq:valthe} is to use the Adams spectral sequence \eqref{eq:ass2} for the homology functor $\pi_*^{\A}$ \cite[Theorem 1.1]{john:torusI}. In our case we obtain a strongly convergent Adams spectral sequence:
\begin{equation}
	\label{eq:adamss}
	\Ext^{s,t}_{\A}(S^V,E\c_G) \Longrightarrow [S^V,E\c_G^{\Sp}]^G_{t-s}=E\c_G^{t-s}(S^V).
\end{equation}
\begin{nota}
	In \eqref{eq:adamss} and for the rest of the section we drop the algebraic model notation and call $S^V:=\pi_*^{\A}(S^V)$ both the spectrum and the corresponding object in $\A(G)$. We also denote $E\c_G^{\Sp}$ the corresponding spectrum of $E\c_{G}$.
\end{nota}
\begin{rem}
	Recall that the index $s$ refers to the $s$-th $\Ext$-group in the graded abelian category $\A(G)$ while $t$ index the grading of these groups, that are graded since the objects $\pi_*^\A(\_)$ are. 
\end{rem}
The bulk of the section will be the computation of the $\Ext$ groups of \eqref{eq:adamss}:
\begin{theorem}
	\label{thm:ext}
	For $V$ as in Theorem \ref{thm:computationspheres} we have the isomorphism of graded groups:
	\begin{equation}
		\label{eq:secondpage}
		\Ext^s_{\A}(S^V,E\c_G)\cong \overline{H^s(\X, \odv)}.
	\end{equation}
	Where on the right in \eqref{eq:secondpage} we have the 2-periodic version of those cohomology groups.
\end{theorem}
As an immediate consequence we obtain \eqref{eq:valthe}:
\begin{proof}[proof of Theorem \ref{thm:computationspheres}]
	The isomorphism \eqref{eq:secondpage} computes the second page of the Adams Spectral sequence \eqref{eq:adamss}. This second page has only the first three rows different from zero (namely for $s=0,1,2$ in \eqref{eq:secondpage}), because the algebraic surface $\X$ has dimension $2$ and therefore its cohomology groups vanishes in higher degrees. Moreover in these three rows we have a chess pattern of zeroes and cohomology groups since the groups \eqref{eq:secondpage} are 2-periodic. This yields the Adams spectral sequence \eqref{eq:adamss} to collapse at the second page, since all the differentials in this page are trivial. In conclusion in this second page we find in the even columns:
	\[
	H^0(\X, \odv)\oplus H^2(\X, \odv)
	\]
	while in the odd ones:
	\[
	H^1(\X, \odv).
	\]
\end{proof}

We devote the rest of this section in proving Theorem \ref{thm:ext}. For this task in \ref{sub:odv} we discuss the Cousin complex of the sheaf $\odv$ needed in \ref{sub:ass} to compute the second page of the Adams spectral sequence. We conclude in \ref{sub:virtualnegative} briefly extending the computation on virtual negative complex representations \eqref{eq:negativevalues}.

\subsection{Cousin complex of $\odv$}
\label{sub:odv}
Here we briefly construct the Cousin complex for $\odv$ exactly as in Sections \ref{sec:2} and \ref{sec:3} for the structure sheaf $\O$. Therefore we recall from those sections how the results change for the sheaf twisted by the line bundle $\odv$. 

Let us start in computing the stalks at the various points $x\in \tp \X$:
\begin{itemize}
	\item If $x=\eta(\tp \X)$ is a generic point of the whole space, in \eqref{eq:stalkgen} nothing changes and we still obtain \eqref{eq:k}:
	\begin{equation}
		\label{eq:cod0stalk}
		\odv_x=\K
	\end{equation}
	\item If $x= \eta(\dij)$ is a generic point of a generating closed subset, by \eqref{eq:dv} $\dij$ appears with coefficient $v_{ij}$ in $D_V$, so \eqref{eq:odij} becomes:
	\begin{equation}
		\label{eq:cod1stalk}
		\odv_{x} =\{\Div(f)\geq v_{ij} \;\text{on} \; \dij \} = \tijvij \O_{\dij}
	\end{equation}
	since $\tij$ defined in \eqref{eq:tij} generates $\mij$.
	\item If $x \in \bvar F$
	(which automatically makes it also a generic point for $\bvar F$), by Lemma \ref{lem:fin} for every direction $i\geq 1$, only $D_{i,n_i}$ contains $\bvar F$ which by \eqref{eq:dv} appears with coefficient $v_{i,n_i}$ in $D_V$. Therefore \eqref{eq:of} becomes:
	\begin{equation}
		\label{eq:isostalkcod2}
		\odv _{x} =\{ f\in \K \mid \Div(f)\geq v_{i,n_i}\, \text{on} \, D_{i,n_i} \, \text{for all} \, i\geq 1 \}=(\prod_{i\geq 1}t^{v_{i,n_i}}_{i,n_i})\O_F
	\end{equation}
\end{itemize}

Proposition \ref{prop:cousincomplex} apply as well to the sheaf $\odv$ and therefore we can consider its Cousin complex, that can be written as the one for $\O$ \eqref{eq:cc2}:

\begin{equation}
	\label{eq:cctwisted}
	\odv \rightarrow \iota_{\X}(\h 0{\X}(\odv)) \xrightarrow{d_0^V} \bigoplus_{ij\geq 1}\iota_{\dij} (\h 1{\dij}(\odv)) \xrightarrow{d_1^V} \bigoplus_{F}\iota_{F}(\h 2F(\odv))
\end{equation}

For each local cohomology term appearing in \eqref{eq:cctwisted} we explicit an isomorphism with the corresponding term in \eqref{eq:cc2}:
\begin{itemize}
	\item In codimension $0$ by \eqref{eq:cod0stalk}:
	\begin{equation}
		\label{eq:h0twistediso}
		\h 0{\X}(\odv)= \odv_{\eta(\tp \X)} = \K
	\end{equation}
	\item In codimension $1$ we have the chain of isomorphisms:
	\begin{equation}
		\label{eq:h1twistediso}
		\h 1{\dij}(\odv)\cong H^1_{\mij}((\odv)_{\eta(\dij)})\cong \frac {\tijvij \O_{\dij}[t^{-1}_{ij}]}{\tijvij \O_{\dij}} = \frac {\K}{\tijvij \O_{\dij}} \cong \frac {\K}{\O_{\dij}}.
	\end{equation}
	The first isomorphism is due by  \eqref{eq:rightderived1} since $\odv$ is the pushforward along $\phi$ of the Zariski twisted sheaf $\zar \ox(-D_V)$. Te second is the computation of local cohomology by means of the stable Koszul complex (\cite{huneke:lectures}), since $\tij$ generates $m_{ij}$, and we have computed the stalk in \eqref{eq:cod1stalk}. The final isomorphism we define it in the following way using the completed coordinates \eqref{eq:hatij}:
	\begin{equation}
		\begin{split}
			 \quotient {\K}{\tijvij \odij}&\xrightarrow{\cong} \quotient {\K}{\odij}\\
			[f] &\mapsto [\hat t^{-v_{ij}}_{ij} \cdot f].
		\end{split}
	\end{equation}
Notice it is well defined exactly as in \eqref{eq:phi0}.
\item In codimension $2$ in the same way we have the chain of isomorphisms: 
	\begin{equation}
		\label{eq:h2twistediso}
		\h 2F(\odv)\cong H^2_{m_F}((\odv)_{\eta(\bvar F)})\cong H^2_{m_F}((\prod_{i\geq 1}t^{v_{i,n_i}}_{i,n_i})\O_F) \cong H^2_{m_F}(\O_F).
	\end{equation}
Like \eqref{eq:h1twistediso} the first isomorphism is due by \eqref{eq:rightderived1} and the second one is the computation of the stalk \eqref{eq:isostalkcod2}. We can define the final isomorphism in the following way:
	\begin{equation}
		\begin{split}
			H^2_{m_F}((\prod_{i\geq 1}t^{v_{i,n_i}}_{i,n_i})\O_F) &\xrightarrow{\cong} H^2_{m_F}(\O_F)\\
			\alpha &\mapsto (\prod_{i\geq 1}\hat t^{-v_{i,n_i}}_{i,n_i})\alpha.
		\end{split}
	\end{equation} 
Notice it is well defined as in Remark \ref{rem:capped}.
\end{itemize}

Now Corollary \ref{cor:exactness} holds as well for $\odv$ since $\zar \ox(-D_V)$ is Cohen-Macaulay with respect to the codimension filtration:
\begin{corollary}
	\label{cor:twistedexactness}
	The Cousin complex \eqref{eq:cctwisted} of $\odv$ is a flabby resolution of $\odv$. 
\end{corollary}

As a consequence we have also Proposition \ref{prop:ccstalk} for $\odv$ where we use the isomorphisms
\eqref{eq:h0twistediso}, \eqref{eq:h1twistediso} and \eqref{eq:h2twistediso} to describe the local cohomology terms:
\begin{proposition}
	\label{prop:twistedccstalk}
	Let $F$ be a finite subgroup of $G$ and $x=\eta (\bvar F)$ be the generic (and only) point of $\bvar F$ in the Kolmogorov quotient $\tp \X$. Then the TP-stalk at $x$ of the Cousin complex \eqref{eq:cctwisted} is the exact sequence:
	\begin{equation}
		\label{eq:twistedccstalk}
		\odv_x \rightarrowtail \K \xrightarrow{d_0^V} \bigoplus_{i\geq 1} \quotient {\K}{\O_{D_{i,n_i}}} \xrightarrow{d_1^V} \h 2F(\O) \to 0.
	\end{equation}
\end{proposition}
\begin{rem}
	Notice \eqref{eq:twistedccstalk} has the same terms as \eqref{eq:ccstalk} but different maps:
	\begin{equation}
		\label{eq:twistd0}
		d_0^V(f)=[\hatij^{-v_{ij}}\cdot f] \in \K/\odij.
	\end{equation}
\begin{equation}
	\label{eq:twistd1}
	d_1^V(\{[f_i]\}_{i})=[\{[f_i\cdot \prod_{s\neq i}\hat{t}_{s,n_s}^{-v_{s,n_s}}]\}_i] \in \h 2F(\O).
\end{equation}
Where we have used \eqref{eq:h2description} to describe $\h 2F(\O)$.
\end{rem}
We conclude with the global sections of the Cousin complex \eqref{eq:cctwisted}:
\begin{equation}
	\label{eq:ccglobsec}
	\Gamma(\odv) \longrightarrow \K \xrightarrow{d_0^V} \bigoplus_{i\geq 1} (\bigoplus_{j\geq 1} \K/\odij) \xrightarrow{d_1^V} \bigoplus_{F} \h 2F(\O) \rightarrow 0
\end{equation} 
Notice that the terms are the same as the global sections \eqref{eq:cc3} for $\O$ but with different maps.
\subsection{Computing the Adams spectral sequence}
\label{sub:ass}
We turn now in computing the second page of the Adams spectral sequence \eqref{eq:adamss}, namely proving Theorem \ref{thm:ext}.

\begin{proof}[Proof of Theorem \ref{thm:ext}]

The aim is to explicitly compute the following sequence:
\begin{equation}
	\label{eq:extcomp}
		\Hom_{\A}(S^V, \inj_0) \xrightarrow{\phi_0'} \Hom_{\A}(S^V, \inj_1) \xrightarrow{\phi_1'} \Hom_{\A}(S^V, \inj_2) \rightarrow 0
\end{equation}
obtained applying the functor $\Hom_{\A}(S^V,\_)$ to the injective resolution \eqref{eq:injres4} of $H_*(E\c_{G})$ (the maps $\phi_0'$ and $\phi_1'$ are the ones induced by $\phi_0$ and $\phi_1$). For this task it is essential the algebraic model $S^V=\pi_*(S^V)\in \A(G)$ that we discuss in detail in the Appendix \ref{sub:spheres}.

By using the adjunction \eqref{eq:adjunction} and the explicit form \eqref{eq:injres1} of the injective resolution of $H_*(E\c_{G})$ we can compute each term of \eqref{eq:extcomp}:
\begin{equation}
	\begin{split}
		\Hom_{\A}(S^V, f_G(\ov \K)) &\cong \Hom_{\Q}(\Q, \ov \K)\cong \ov \K \\
		\Hom_{\A}(S^V, \bigoplus_{i\geq 1}f_{H_i}(T_i)) &\cong \bigoplus_{i\geq 1}\Hom_{\oefh i}(\susp {V^{H_i}}\oefh i T_i)\cong \bigoplus_{i\geq 1}\susp {-V^{H_i}}T_i, \\
		\Hom_{\A}(S^V, f_1(N)) &\cong \Hom_{\of}(\susp {V}\of, N)\cong \susp {-V}N
	\end{split}
\end{equation}
where for the second equation we obtain the direct sum for $i\geq 1$ instead of the product since $S^V$ is a small object. 

The sequence \eqref{eq:extcomp} than takes the form
\begin{equation}
	\label{eq:homcomp}
		\ov \K \xrightarrow{\phi_0'} \bigoplus_{i\geq 1} \susp {-V^{H_i}}T_i \xrightarrow{\phi_1'}  \susp {-V}N \rightarrow 0
\end{equation}
with maps:
\begin{equation}
	\label{eq:phi0prime}
	\begin{split}
		\phi_0'(f) &=\{\phi_0^i(e(V^{H_i})^{-1}\otimes f)\}_{i\geq 1} \\
		\phi_1'(\{\alpha_i\}_{i\geq 1}) &=\phi_1(\{e(V-V^{H_i})^{-1}\tens i\alpha_i\}_i)
	\end{split}
\end{equation}
where $\phi_0$ and $\phi_1$ satisfy the commutative diagrams \eqref{eq:comm0} and \eqref{eq:comm1} and the Euler classes $e(V^{H_i})^{-1}\in \eghinv i\oefh i$ and $e(V-V^{H_i})^{-1}\in \ehinv i \of$ pop out from the structure maps of the object $S^V$ \eqref{eq:structuremaps}.

The sequence \eqref{eq:homcomp} is $2$-periodic, therefore we can work at the  $0$-th level. Note as well that the desuspensions in the second and third term shift $2$-periodic objects by an even degree, therefore the term at each level does not change. The $0$-th level of \eqref{eq:homcomp} is:
\begin{equation}
		\label{eq:homcomp2}
		\K \xrightarrow{\phi_0'} \bigoplus_{i\geq 1} (\bigoplus_{j\geq 1} \K/\odij) \xrightarrow{\phi_1'}  \bigoplus_{F} \h 2F(\O) \rightarrow 0.
\end{equation}

To conclude, we only need to prove that \eqref{eq:homcomp2} is the sequence of global sections \eqref{eq:ccglobsec} of the Cousin complex of $\odv$, since by Corollary \eqref{cor:twistedexactness} the Cousin complex is a flabby resolution and therefore the homology of its global sections \eqref{eq:homcomp2} gives us the desired cohomology groups $H^*(\X, \odv)$ (Recall from Corollary \ref{cor:coh} that the cohomology of $\odv$ is the same for both topologies).

We only need to check that the maps in \eqref{eq:homcomp2} and \eqref{eq:ccglobsec} are the same. By \eqref{eq:phi0prime} the map $\phi_0'$ on every component $(i,j)$ is the map:
\begin{equation*}
	\phi_0'(f)=\phi_0(c_{ij}^{-v_{ij}}\otimes f)=[ \hatij^{-v_{ij}}\cdot f] = d_0^V(f),
\end{equation*}
where the second equality follows by definition of $\phi_0$ \eqref{eq:phi0}, and the last one is \eqref{eq:twistd0}.

By \eqref{eq:phi0prime} the $F$th component of $\phi_1'$ for every finite subgroup $F$ of $G$ is the map:
\begin{equation*}
	\phi_1'(\{[f_i]\}_{i}) = \phi_1(\{ \prod_{s\neq i}x_s^{-v_{s,n_s}} \tens i [f_i]\}_i)= [\{[f_i\cdot \prod_{s\neq i}\hat{t}_{s,n_s}^{-v_{s,n_s}}]\}_i] = d_1^V(\{[f_i]\}_{i}),
	\end{equation*}
where the second equality follows by definition of $\phi_1$ \eqref{eq:phi1}, the last one is \eqref{eq:twistd1}, and we have used \eqref{eq:h2description} to describe elements in $\h 2F(\O)$.
\end{proof}
\subsection{Virtual negative complex representations}
\label{sub:virtualnegative}
All the computations presented in this section work exactly in the same way for virtual negative complex representations. If $V$ is a genuine complex representation with $V^G=0$, Theorem \ref{thm:computationspheres} changes sign:
\begin{equation}
	\label{eq:negativevalues}
	E\c_{G}^n(S^{-V})\cong 
	\begin{cases*}
		H^0(\X, \O(D_V))\oplus H^2(\X, \O(D_V)) & $n$ even \\
		H^1(\X, \O(D_V)) & $n$ odd.
	\end{cases*}.
\end{equation}

Everything starts with the changes in the object $S^{-V}\in \A(G)$ explained in Remark \ref{rem:appendixnegative} where the desuspensions change sign into suspensions, and therefore \eqref{eq:homcomp} changes into:
\begin{equation*}
	\label{eq:homcompvirt}
	\ov \K \xrightarrow{\phi_0'} \bigoplus_{i\geq 1} \susp {V^{H_i}}T_i \xrightarrow{\phi_1'}  \susp {V}N \rightarrow 0
\end{equation*}
where in the maps \eqref{eq:phi0prime} positive powers of the Euler classes appear instead of the negative ones. As a consequence Theorem \ref{thm:ext} changes sign:
\[
\Ext^s_{\A}(S^{-V},E\c_G)\cong \overline{H^s(\X, \O(D_V))}.
\]
and we conclude exactly as before.

\appendix
\section{Algebraic models}
\label{appendix:algmodel}
In this appendix we present a self-contained account of algebraic models for tori of any rank, therefore  only in this appendix $G=\T^r$ is an r-dimensional torus, with $r\geq 0$. We also specify that all the modules over graded rings are graded and all the maps between graded modules are graded maps.

Algebraic models are a useful tool to study rational equivariant cohomology theories. The main idea is to define an abelian category $\A(G)$ and an homology functor:
\begin{equation}
	\label{eq:homologyfunctor2}
	\pi_*^{\A}:\sp G_{\Q} \to \A(G)
\end{equation}
equipped with an Adams spectral sequence to compute maps in rational $G$-spectra. More precisely the values of the theory may be calculated by a spectral sequence:
\begin{equation}
	\label{eq:ass2}
	\Ext^{*,*}_{\A}(\pi_*^\A(X),\pi_*^\A(Y)) \Longrightarrow [X,Y]^G_{*}.
\end{equation}
In the case of tori the homology functor $\pi_*^{\A}$ can be lifted to a Quillen-equivalence \cite[Theorem 1.1]{john:torusship}):
\begin{theorem}[Greenlees-Shipley]
	\label{thm:shipley}
	The category $d\A(G)$ of differential graded objects  in $\A(G)$ is Quillen-equivalent to rational $G$-spectra.
\end{theorem}

When this happens algebraic models can also be used to build rational $G$-equivariant cohomology theories simply constructing objects in $d\A(G)$.

\subsection{Definition of the rings}
We start by defining the rings needed for the construction of $\A(G)$ \cite[Section 3.A.]{john:torusI}. We write $\F$ for the family of finite subgroups of $G$.

\begin{definition}
	For every connected subgroup $H$ of $G$ define the collection:
	\begin{equation*}
		\quotient {\F}H:=\{\tilde{H}\leq G \mid H  \text{ finite index in } \tilde H \}
	\end{equation*}
and the ring:
	\begin{equation}
		\label{eq:ofhi}
		\oef H := \prod_{\tilde H \in \F/H}H^*\left(B\left( \quotient G{\tilde H}\right)\right)
	\end{equation}
\end{definition}

\begin{rem}
	Note $\oef G=\Q$ and $\oef 1=\of$.
\end{rem}
Any containment of connected subgroups $K\subseteq H$ induces an inflation map $\oef H \to \oef K$, defined in the following way. 
\begin{definition}
	The inclusion $K\subseteq H$ of connected subgroups defines a quotient map $q:\quotient GK \to \quotient GH$, and hence
	\begin{equation}
		\label{eq:qstar}
		q_*:\quotient \F{K}\to \quotient \F{H}.
	\end{equation}
	For any $\tilde{K}\in \quotient \F{K}$ define the $\tilde{K}$th component of the inflation map $\oef H \to \oef K$ to be the composition:
	\begin{equation}
		\label{eq:inflationmap}
		\oef H=\prod_{\tilde H \in \F/H}\hbg {\tilde H} \to \hbg {q_*\tilde{K}} \to \hbg {\tilde K}
	\end{equation}
	given by projection onto the term $\hbg {q_*\tilde{K}}$ followed by the inflation map induced by the quotient $\quotient G{\tilde K} \to \quotient G{q_*\tilde K}$.
	
\end{definition}
\begin{rem}
	\label{rem:inflationmap}
	In particular for any connected subgroup $H$ we have an inflation map induced by the inclusion of the trivial subgroup:
	\begin{equation}
		\label{eq:inflation}
		i_H:\oef H \to \of 
	\end{equation}
	which is a split monomorphism of $\oef H$-modules \cite[Proposition 3.1]{john:torusII}. As a consequence $\of$ is an $\oef H$-module for every connected subgroup $H$.
\end{rem}

\subsection{Euler classes}
\label{sub:eulerclass}
Fundamental elements of these rings are Euler classes of representations of $G$, used in the localization process. For any complex representation $V$ of $G$ we want to define its Euler class $e(V)\in \of$ \cite[Section 3.B.]{john:torusI}. We require them to be multiplicative: $e(V\oplus W)=e(V)e(W)$, therefore it's enough to define Euler classes for one dimensional complex representations $V$.
\begin{definition}
	For a one dimensional complex representation $V$ of $G$, define its Euler class $e(V)\in \of$ as follows. For every finite subgroup $F$ the $F$th component $e(V)_F\in \hbg F$ is:
	\[
	e(V)_F=
	\begin{cases*}
		1 & if $V^F=0$ \\
		\bar e(V^F) & if $V^F\neq 0$,
	\end{cases*}
	\]
\end{definition}
where $\bar e(V^F)\in H^2(BG/F)$ is the classical equivariant Euler class for the $G/F$ representation $V^F$.

\begin{definition}
	For any connected subgroup $H$ of $G$ define the multiplicatively closed subset of $\of$:
	\begin{equation}
		\label{eq:eh}
		\E_H:=\{e(V)\mid V^H=0\}.
	\end{equation}
\end{definition}

\subsection{The $2$-torus}
We are mainly interested in the case of the 2-torus, therefore let us compute explicitly rings and Euler classes in this case. Recall that $\{H_i\}_{i\geq 1}$ is the collection of connected closed codimension $1$ subgroups of $\T^2$ and $H_i^j$ is the subgroup with $j$-components and identity component $H_i$. Exactly as in Definition \ref{def:zij} $z_i^j$ is a character of $\T^2$ with kernel $H_i^j$. In this case:
	\begin{itemize}
		\item $\oef {\T^2}=\Q $
		\item $\oefh i= \prod_{j\geq 1} H^*(B\T^2/{H_i^j})$
		\item $\of= \prod_{F} H^*(B\T^2/F)$ , where $F$ runs through all the finite subgroups of $\T^2$.
	\end{itemize}
For every $i,j\geq 1$:
\begin{equation}
	\label{eq:cij}
	H^*(B\T^2/{H_i^j})\cong \Q[c_{ij}]
\end{equation}  
where $c_{ij}=e(z_i^j)$ of degree $-2$ is the Euler class of the character $z_i^j$ (more precisely of the one dimensional complex representation of $\T^2$ defined by the character $z_i^j$).

\begin{definition}
	\label{def:ni}
	For every finite subgroup $F$ of $\T^2$ and every index $i\geq 1$ define $n_i=n_i(F)$ to be the only positive integer such that $H_i^{n_i}$ is generated by $H_i$ and $F$: $\gen {F,H_i}=H_i^{n_i}$.
\end{definition}

Every finite subgroup $F$ can be written as the intersection of two codimension one subgroups of $\T^2$. Therefore for every $F$ there exists two different integers $A=A(F)\geq 1$ and $B=B(F)\geq 1$ such that
\begin{equation}
	\label{eq:splitting}
	F=H_A^{n_A}\cap H_B^{n_B}.
\end{equation}
\begin{choice}
	For any finite subgroup $F$ we choose a couple of positive integers $(A,B)$ that give the decomposition \eqref{eq:splitting}.
\end{choice}

By \eqref{eq:splitting} we obtain the decomposition:
\begin{equation}
	\label{eq:xaxb}
	H^*(B\T^2/F)\cong H^*(B\T^2/{\ha})  \otimes H^*(B\T^2/{\hb}) \cong \Q[\xa,\xb].
\end{equation} 
Where $\xa:=e(\za)$, $\xb:=e(\zb)$ have both degree $-2$ are the Euler classes respectively of $\za$ and $\zb$.
\begin{definition}
	\label{def:xi}
	For every $i\geq 1$ define
	\begin{equation}
		\label{eq:xi}
		x_i:=e(z_i^{n_i})\in H^2(B\T^2/F).
	\end{equation} 
	Notice it is an integral linear combination of $\xa$ and $\xb$.
\end{definition}

\begin{rem}
\label{rem:inflationcomp}
	With these choices of coordinates \eqref{eq:cij} and \eqref{eq:xaxb} the inflation map \eqref{eq:inflation} on the $F$th component of the target $\of$ can be easily described:
	\begin{equation}
	\label{eq:inflationhi}
	\oefh i = \prod_{j\geq 1}H^*(B\T^2/{H_i^j}) \to H^*(B\T^2/{H_i^{n_i}}) \rightarrowtail H^*(B\T^2/F).
	\end{equation}
	The first map of \eqref{eq:inflationhi} is the projection onto the $n_i$th component since $q_*(F)=\gen {H_i,F}=H_i^{n_i}$ by definition of the index $n_i$. The second map of \eqref{eq:inflationhi} is the natural inclusion of $\Q$-algebras sending the generator $c_{i,n_i}$ to $x_i$, since by \eqref{eq:xi} they are the same Euler class $e(z_i^{n_i})$ for the two different rings.
\end{rem}

\subsection{Description of $\A(G)$}
We briefly recap the description of $\A(G)$ \cite[Definition 3.9]{john:torusI}. The objects of $\A(G)$ are sheaves of modules over the poset of connected subgroups of $G$ with inclusions.

\begin{definition}
	An object $X\in \A(G)$ is specified by the following pieces of data:
	\begin{enumerate}
		\item For every connected subgroup $H$ an $\oef H$-module $\phi^HX$.
		\item For every containment of connected subgroups $K\subseteq H$ an $\oef K$-modules map:
		\begin{equation}
			\label{eq:strmap}
			\phi^KX\to \E_{\quotient HK}^{-1}\oef K\tens {\oef H}\phi^HX.
		\end{equation}
	\end{enumerate}
	Then $X$ is a sheaf over the space of connected subgroups of $G$. This specifically means that for every connected subgroup $H$, the sheaf $X$ has value the $\of$-module:
	\begin{equation}
		\label{eq:valuex}
		X(H):=\E_H^{-1}\of \tens {\oef H} \phi^HX,
	\end{equation}
	and that for every containment $K\subseteq H$ of connected subgroups, $X$ has a structure map of $\of$-modules:
	\begin{equation}
		\label{eq:strmapofx}
		\beta_K^H:X(K) \to X(H).
	\end{equation}
	The map \eqref{eq:strmapofx} is obtained tensoring the $\oef K$-modules map \eqref{eq:strmap} with the $\oef K$-module $\E_K^{-1}\of$. Moreover $X$ satisfies the condition that for every connected subgroup $H$ the $\of$-modules structure map $\beta_1^H:X(1)\to X(H)$ is the map inverting the multiplicatively closed subset of Euler classes $\E_H$ \eqref{eq:eh}.
\end{definition}
\begin{rem}
	Notice that by Remark \ref{rem:inflationmap} the inflation map $i_K$ makes $\of$ an $\oef K$-module. Moreover The structure map \eqref{eq:strmap} is well defined from \eqref{eq:strmap}, since:
	\begin{equation}
		\label{eq:strtensored}
		\E_K^{-1}\of\tens {\oef K} \E_{\quotient HK}^{-1}\oef K \cong \E_H^{-1}\of.
	\end{equation}
\end{rem}

\begin{ex}
	For the 2-torus an object $X\in \A(\T^2)$ has the shape:
	\begin{equation*}
		\begin{tikzcd}
			& X(\T^2) &  \\
			X(H_1) \arrow[ur, "\beta_{H_1}^{\T^2}"] & X(H_2) \arrow[u] & \dots \arrow[ul] \\
			& X(1) \arrow[ur] \arrow[u] \arrow[ul, "\beta_1^{H_1}"] & 
		\end{tikzcd}=
		\left[ \begin{tikzcd}
			X(\T^2) \\
			X(H_i) \arrow[u] \\
			X(1) \arrow[u]  
		\end{tikzcd}\right] 
	\end{equation*}
with infinite values $X(H_i)$ in the middle row, one vertex $X(\T^2)$ and one value $X(1)$ at the bottom level. By \eqref{eq:valuex}: 
\begin{equation}
	\begin{split}
		X(\T^2) & =\etinv \tens {\Q} \phi^{\T^2}X \\
		X(H_i) & =\ethinv i \of \tens {\oefh i}\phi^{H_i}X \\
		X(1) & =\of \tens {\of} \phi^1X = \phi^1X
	\end{split}
\end{equation}
\end{ex}
\begin{nota}
    \label{notation:tensi}
	a tensor product with no ring specified will always mean over $\Q$. For any $i\geq 1$ we denote $\tens i=\tens {\oefh i}$ the tensor product over the ring $\oefh i$ or when we are considering the $F$th component: $\tens i=\tens {\hbti {n_i}}$.
\end{nota}
\begin{ex}
	\label{ex:eulerclasses}
	For the 2-torus using the coordinates we have defined (\eqref{eq:cij}, \eqref{eq:xaxb}, and  \eqref{eq:xi}), we can easily describe the localizations at the Euler classes:
	\begin{itemize}
		\item In $\eghinv i \hbti j$ we are inverting the Euler class $c_{ij}$.
		\item In $\ehinv i\hbt F$ we are inverting all the Euler classes $x_j$ with $j\geq 1$ and $j\neq i$.
		\item In $\etinv \hbt F$ we are inverting all the Euler classes $x_j$ with $j\geq 1$. 
	\end{itemize}
\end{ex}
\begin{ex}
	There is a structure sheaf $\O\in \A(G)$ \cite[Definition 3.3]{john:torusI} obtained using as modules the base rings: $\phi^H\O=\oef H$, and as structure maps the natural inclusions:
	\[
	\oef K \rightarrow \E_{\quotient HK}^{-1}\oef K\tens {\oef H}\oef H.
	\]
\end{ex}
\begin{definition}
	A morphism $f:X\to Y$ in the category $\A(G)$ is the data of a (graded) $\oef H$-module map $\phi^Hf:\phi^HX \to \phi^HY$ for every connected subgroup $H$, compatible with the structure maps of $X$ and $Y$ (it makes the evident commutative diagrams between different levels commute \cite[Definition 3.6]{john:torusI}).
\end{definition}
\begin{rem}
	\label{rem:bottomlevel}
	A morphism $f:X\to Y$ in $\A(G)$ is almost determined by what it does at the trivial subgroup level $f(1):X(1)\to Y(1)$ (that we will call bottom level). This is because for any connected subgroup $H$ the map $f$ at the $H$th level $f(H):X(H)\to Y(H)$ is then $\einv H f(1)$. Therefore properties like injectivity, surjectivity or exactness for a sequence of morphisms can be checked at the bottom level.
\end{rem}

\subsection{Injectives in $\A(G)$}
The injective objects in $\A(G)$ that we will use are constant below a certain connected subgroup $H$, and zero elsewhere \cite[Section 4.A.]{john:torusI}.
\begin{definition}
	An object $X\in \A(G)$ is concentrated below a connected subgroup $H$ if $X(K)=0$ for every connected subgroup $K\nsubseteq H$. We denote $\A(G)_H$ the full subcategory of $\A(G)$ of objects concentrated below $H$.
\end{definition}
\begin{definition}
	If $H$ is a connected subgroup of $G$, and $T$ is a graded torsion $\oef H$-module, define $f_H(T)\in \A(G)$ to be the constant sheaf below $H$ with the following values:
	\begin{equation}
		\label{eq:deffh}
	f_H(T)(K):=
	\begin{cases*}
		\einv H\of \tens {\oef H} T & if $K \subseteq H$ \\
		0 & if $K \nsubseteq H$
	\end{cases*}
	\end{equation}
	 and structure maps either identities or zero.
\end{definition}
\begin{rem}
	We require $T$ to be torsion so that when we invert $\einv K$ for $K\nsubseteq H$ we obtain zero. Therefore this requirement can be dropped when $H=G$.
\end{rem}
\begin{lemma}[Lemma 4.1 of \cite{john:torusI}]
	\label{lem:adj}
	For any connected subgroup $H$ of $G$ there is an adjunction:
	\begin{equation*}
		\begin{tikzcd}
			\A(G)_H\ar[r,bend left,"\phi^{H}",""{name=A, below}] & 
			\quad \quad\text{Tors-$\oef H$-Mod} \ar[l,bend left,"f_H",""{name=B,above}] \ar[from=A, to=B, symbol=\dashv]
		\end{tikzcd}
	\end{equation*}
	where the left adjoint is the evaluation $\phi^H$. Moreover for any torsion $\oef H$-module $T$, and object $X\in \A(G)$ we have:
	\begin{equation}
		\label{eq:adjunction}
		\Hom_{\oef H}(\phi^HX, T)\cong \Hom_{\A(G)}(X, f_H(T)).
	\end{equation}
\end{lemma}
This in order allows us to transfer torsion injectives $\oef H$-modules into injective objects in $\A(G)$. First notice that if we are given for every $\tilde H\in \quotient {\F}H$ an $\hbg {\tilde H}$-torsion module $T(\tilde H)$, then $\bigoplus_{\quotient {\F}H}T(\tilde H)$ is naturally a torsion $\oef H$-module, with the action given component by component.
\begin{corollary}[Lemma 5.1 of \cite{john:torusI}]
	\label{cor:inj1}
	Suppose $\{H_i\}_{i\geq 1}$ is the collection of all connected subgroups of $G$ of a fixed dimension. If $f_{H_i}(T_i)$ is injective for every $i$, then so is $\bigoplus_{i\geq 1}f_{H_i}(T_i)$.
\end{corollary}
\begin{corollary}[Corollary 5.2 of \cite{john:torusI}]
	\label{cor:inj2}
	If for every $\tilde H \in \quotient {\F}H$, $T(\tilde H)$ is a graded torsion injective $\hbg {\tilde H}$-module. Then $f_H(\bigoplus_{\quotient {\F}H}T(\tilde H))$ is injective in $\A(G)$.
\end{corollary}

\subsection{Spheres of complex representations}
\label{sub:spheres}
We can now define the fundamental homology functor $\pi_*^{\A}$ \cite[Definition 1.4]{john:torusI}. Given a rational $G$-spectrum $X$ we can define the sheaf $\pi_*^{\A}(X)\in \A(G)$ that on a connected subgroup $H$ takes the value
\[
\pai \A(X)(H):=\einv H\of \tens {\oef H}\pi_*^{G/H}(\defhp H \wedge \fix HX).
\]
Where $\fix H$ is the geometric fixed point functor, $E\F_+$ is the universal space for the family $\F$ of finite subgroups with a disjoint basepoint added, and $\defp=F(E\F_+, S^0)$ is its functional dual (The function spectrum of maps from $E\F_+$ to $S^0$). 
\begin{ex}
	For the sphere spectrum we obtain the structure sheaf $\O$ \cite[Theorem 1.5]{john:torusI}:
	\[
	\pai \A(S^0)(H) =\einv H\of \tens {\oef H}\oef H
	\]
\end{ex}
Given a complex representation $V$ of $G$ with $V^G=0$ we want to make explicit the object $\pai \A(S^V)$ \cite[Section 2.B.]{john:torusII}. The value at each subgroup is computed with fixed points of the representation:
\begin{equation}
	\label{eq:sphere}
	\pai \A(S^V)(H) =\einv H\of \tens {\oef H}\Sigma^{V^H}\oef H.
\end{equation}
To describe the structure maps it is convenient to use the suspensions of the units:
\begin{equation*}
	\i {V^H} :=\Sigma^{V^H}(1)\in \Sigma^{V^H}\oef H
\end{equation*}
so that for every inclusion of connected subgroups $K\subseteq H$ the structure map $\beta_K^H$ is determined by the suspended unit:
\begin{equation}
	\label{eq:structuremaps}
	\beta_K^H(\i {V^K})= e(V^K-V^H)^{-1}\otimes \i {V^H}
\end{equation}
where the difference of the two representations simply means a splitting
\[
V^K=V^H\oplus (V^K-V^H).
\]
\begin{rem}
	\label{rem:appendixnegative}
	The content of this section applies also in the case of a virtual complex representation $V=V_0-V_1$. The only thing to specify is the Euler class $e(V)=e(V_0)/e(V_1)$. As a result in \eqref{eq:sphere} we have the desuspension  $\Sigma^{-V_1^H}\oef H$ instead of the suspension.
\end{rem}


\nocite{*}
\printbibliography

\end{document}